\documentclass{article}
\usepackage{amssymb}

\usepackage{arxiv}
\usepackage{array}
\usepackage{url}
\usepackage{hyperref}
\usepackage{amsmath,amsthm,amssymb,amsbsy}
\usepackage{aliascnt}
\usepackage{mathtools}
\usepackage{paralist}
\usepackage[dvipsnames]{xcolor}
\usepackage{color}
\usepackage{graphicx}
\usepackage{comment}
\usepackage{wrapfig}
\usepackage{caption}
\usepackage{subcaption}
\usepackage{bm}
\usepackage{framed}
\usepackage{fancyhdr}
\usepackage{enumitem}
\usepackage[T1]{fontenc}
\usepackage[utf8]{inputenc}

\usepackage[ruled, linesnumbered]{algorithm2e}
\SetKwInput{kwInit}{Initialization}

\usepackage[capitalize,nameinlink]{cleveref}[0.19]
\usepackage{booktabs} 
\usepackage{wrapfig}
\usepackage{hyperref}
\usepackage{microtype}     
\usepackage{tcolorbox}
\usepackage{tikz}
\usetikzlibrary {arrows.meta,graphs,shapes}

\definecolor{mydarkblue}{rgb}{0,0.08,0.45} 
\hypersetup{
  colorlinks=true,
  frenchlinks=false,
  pdfborder={0 0 0},
  naturalnames=false,
  hypertexnames=false,
  breaklinks,
  linkcolor=mydarkblue,
  citecolor=mydarkblue,
  filecolor=mydarkblue,
  urlcolor = black,
}

\numberwithin{equation}{section} 

\newtheorem{theorem}{Theorem}[section] 
\newaliascnt{lemma}{theorem}
\newtheorem{lemma}[lemma]{Lemma} 
\aliascntresetthe{lemma}
\newaliascnt{corollary}{theorem}
\newtheorem{corollary}[corollary]{Corollary}
\aliascntresetthe{corollary}
\newaliascnt{proposition}{theorem}
\newtheorem{proposition}[proposition]{Proposition}
\aliascntresetthe{proposition}
\newaliascnt{definition}{theorem}

\aliascntresetthe{definition}
\newaliascnt{problem}{theorem}

\aliascntresetthe{problem}
\newaliascnt{observation}{theorem}

\aliascntresetthe{observation}
\newaliascnt{assumption}{theorem}
\newtheorem{assumption}[assumption]{Assumption}
\aliascntresetthe{assumption}
\newaliascnt{fact}{theorem}

\aliascntresetthe{fact}

\theoremstyle{definition}

\crefname{theorem}{theorem}{theorems}
\Crefname{theorem}{Theorem}{Theorems}
\crefname{lemma}{lemma}{lemmas}
\Crefname{lemma}{Lemma}{Lemmas}
\crefname{corollary}{corollary}{corollaries}
\Crefname{corollary}{Corollary}{Corollaries}
\crefname{proposition}{proposition}{propositions}
\Crefname{proposition}{Proposition}{Propositions}
\crefname{definition}{definition}{definitions}
\Crefname{definition}{Definition}{Definitions}
\crefname{problem}{problem}{problems}
\Crefname{problem}{Problem}{Problems}
\crefname{observation}{observation}{observations}
\Crefname{observation}{Observation}{Observations}
\crefname{assumption}{assumption}{assumptions}
\Crefname{assumption}{Assumption}{Assumptions}
\crefname{fact}{fact}{facts}
\Crefname{fact}{Fact}{Facts}
\crefname{remark}{remark}{remarks}
\Crefname{remark}{Remark}{Remarks}

\newcommand{\munderbar}[1]{\underline{#1}}

\newcommand{\RR}{{\sf RR }}
\newcommand{\RRsc}{{\sf RR}\text{-}{\sf sc}}

\newcommand{\sL}{{\sf L}}
\newcommand{\sG}{{\sf G}}
\newcommand{\sA}{{\sf A}}
\newcommand{\sB}{{\sf B}}
\newcommand{\sC}{{\sf C}}

\newcommand{\sF}{{\sf F}}
\newcommand{\scc}{{\sf sc}}

\newcommand{\e}{\begin{equation}}
\newcommand{\ee}{\end{equation}}
\newcommand{\en}{\begin{equation*}}
\newcommand{\een}{\end{equation*}}
\newcommand{\eqn}{\begin{eqnarray}}
\newcommand{\eeqn}{\end{eqnarray}}
\newcommand{\bmat}{\begin{bmatrix}}
\newcommand{\emat}{\end{bmatrix}}
\newcommand{\btab}{\begin{tabular}}
\newcommand{\etab}{\end{tabular}}

\newcommand{\op}{{\operatorname{op}}}

\graphicspath{{./fig/}}

\title{High Probability Guarantees for Random Reshuffling}

\author{
    Hengxu Yu \\
    School of Data Science \\
    The Chinese University of Hong Kong, Shenzhen \\
    \texttt{hengxuyu@link.cuhk.edu.cn}
    \And
    Xiao Li \thanks{Xiao Li was supported in part by the National Natural Science Foundation of China (NSFC) under grants 12571330 and 12201534, and in part by the 1+1+1 CUHK-CUHK(SZ)-GDSTC Joint Collaboration Fund under grant 2025A0505000049.} \\
    School of Data Science \\
    The Chinese University of Hong Kong, Shenzhen \\
    \texttt{lixiao@cuhk.edu.cn}
}

\begin{document}

\maketitle

\begin{abstract}
We consider the stochastic gradient method with random reshuffling ($\mathsf{RR}$) for tackling smooth nonconvex optimization problems. $\mathsf{RR}$ finds broad applications in practice, notably in training neural networks. In this work, we provide high probability complexity guarantees for this method. First, we establish a \emph{high probability} ergodic sample complexity result (without taking expectation) for finding an $\varepsilon$-stationary point. Our derived complexity matches the best existing in-expectation one up to a logarithmic term while imposing no additional assumptions nor modifying $\mathsf{RR}$'s updating rule. Second, building on this analysis, we propose a simple \emph{stopping criterion} embedded with a computable stopping test for $\mathsf{RR}$ (denoted as $\mathsf{RR}$-$\mathsf{sc}$). This criterion is guaranteed to be triggered after a finite number of iterations, enabling us to prove the same order high probability complexity for the returned last iterate. The fundamental ingredient in deriving the aforementioned results is a new \emph{concentration property for random reshuffling}, which could be of independent interest. Finally, we conduct numerical experiments on small neural network training to support our theoretical findings.
\end{abstract}


\section{Introduction}\label{sec:intro}
In this work, we study the following smooth nonconvex finite-sum optimization problem:
\begin{equation}\label{eq:problem}
  \min _{x \in \mathbb{R}^d } f(x) = \frac{1}{n} \sum_{i=1}^{n} f_i(x),
\end{equation}
where each component function $f_i$ is continuously differentiable, though not necessarily convex. This form of optimization problem is ubiquitously found in various engineering fields, including machine learning and signal processing \cite{bottou2018optimization,chi2019nonconvex}. 

\textbf{Assumption.} Throughout this paper, we make the following standard smoothness assumption on the component functions.

\begin{assumption}\label{Assumption:L smooth}
For all $i\in [n]$, $f_i$ in \eqref{eq:problem} is bounded from below by $\bar f_i$ and its gradient $\nabla f_i $ is Lipschitz continuous with parameter $\sL> 0$.
\end{assumption}

Let $\bar f$ be a lower bound (e.g., optimal value) of $f$ in \eqref{eq:problem}.  It was established in \cite[Proposition 3]{khaled2022better} that the following variance-type bound is true once \Cref{Assumption:L smooth} holds:
\begin{equation}\label{eq:ABC condition}
    \frac{1}{n} {\sum}_{i=1}^n \|\nabla f_i(x) - \nabla f(x)\|^2 \leq \sA(f(x) - \bar f) + \sB,
\end{equation}
where $\sA = 2\sL > 0$ and $\sB = \frac{\sA}{n} \sum_{i=1}^n (\bar f - \bar f_i) \geq0$.
This variance-type inequality can be traced back to \cite{poljak1973pseudogradient} for providing a unified analysis of various stochastic optimization method.

Let us remark that the lower boundedness condition on each component function $f_i$ in \Cref{Assumption:L smooth} is only used for deriving an explicit form of the constant $\sB$ in \eqref{eq:ABC condition}. The property \eqref{eq:ABC condition} may still hold for some constant $\sB$ even without this condition. In this case, the lower boundedness assumption on each $f_i$ is no longer required.

\textbf{The random reshuffling method.} Many contemporary real-world applications of form \eqref{eq:problem} are large-scale. A notable example of such a scenario is the training of deep neural networks. In the large-scale case, a popular method for addressing problem \eqref{eq:problem} is the stochastic gradient method (SGD) \cite{robbins1951stochastic,ghadimi2013stochastic}, which adopts a uniformly random sampling of the component functions with replacement. Despite SGD being studied extensively in theory over the past decades, the variant commonly implemented in practice for tackling \eqref{eq:problem} is the stochastic gradient method with random reshuffling ($\RR$); see, e.g., \cite{bertsekas2011,bottou2012,haochen2019random,gurbu2019,sun2020optimization}. In the following, we review the algorithmic scheme of $\RR$.

To accommodate the large-scale nature of the contemporary applications, $\RR$ implements a gradient descent-type scheme, but in each update it uses only one (or a minibatch) of the component functions rather than all the components. To describe the algorithmic scheme of $\RR$, we define the set of all possible permutations of $\left\{ 1,2, \dots ,n\right\} $ as
\begin{equation}\label{eq:permutations}
   \Pi := \left\{ \pi : \pi \text{ is a permutation of }  \left\{ 1,2, \dots ,n\right\}  \right\}.
\end{equation}
At the $t$-th iteration/epoch\footnote{By convention in $\RR$, an epoch is also called an iteration. We use both names interchangeably in this paper. By contrast, we call each inner update step an inner iteration.}, $\RR$ first samples a permutation $\pi_{t}$ from $\Pi$ uniformly at random. Then, it starts with an initial inner iterate $x_t^0 = x_t$ and updates $x_t $ to $x_{t+1}$ by consecutively applying the gradient descent-type steps as
\begin{equation}\label{eq:RR update}
  x_t^{i} = x_t^{i-1} -\alpha_t \nabla f_{\pi _t^{i} } (x_t^{i-1})
\end{equation}
for $i = 1,\ldots, n$, yielding $x_{t+1} = x_t^n$. We display the pseudo code of $\RR$ in \Cref{alg:RR}.

Let us mention that the deterministic counterpart of $\RR$, namely, the incremental gradient method, is also widely used in practice and has received considerable attention in the past decades; see, e.g., \cite{nedic2001incremental,bertsekas2011,gurbuzbalaban2015,mokhtari2018surpassing} and the references therein. 

\IncMargin{1em}
\begin{algorithm}[t]
	\caption{$\RR$: Random Reshuffling}\label{alg:RR}
	\KwIn{Initial point $x_0\in \mathbb{R}^{d}$}
	\For{$t= 0,1,\ldots$}{
		Sample $\pi_t = \{\pi_t^{1} , \dots, \pi_t^{n}\}$ uniformly at random from $\Pi$ defined in \eqref{eq:permutations}\;
		Update the step size $\alpha_t$ according to a certain rule\;
		Set $x_t^0 = x_t$\;
		\For{$i = 1, \ldots, n$}{
			$x_t^{i} = x_t^{i-1} -\alpha_t \nabla f_{\pi _t^{i} } \left(x_t^{i-1} \right)$\tcc*[r]{update}
		}
		Set $x_{t+1}= x_t^{n} $\;
	}
\end{algorithm}
\DecMargin{1em}

The primary difference between $\RR$ and SGD lies in that the former employs a uniformly random sampling without replacement. Therefore, $\RR$ is also known as ``SGD without replacement", ``SGD with reshuffling", ``shuffled SGD", etc. This sampling scheme introduces statistical dependence and removes the unbiased gradient estimation property used in SGD, making its theoretical analysis more challenging. Nonetheless, $\RR$ empirically outperforms SGD \cite{bottou2009curiously,recht2013parallel} and the gradient descent method \cite{bertsekas2011} on many practical problems. Such a superior practical performance over SGD arises partly from the fact that the random reshuffling sampling scheme is simpler and faster to implement than sampling with replacement used in SGD, and partly from the property that $\RR$ utilizes all the training samples at each iteration. Owing to these advantages, $\RR$ has been incorporated into prominent software packages like PyTorch and TensorFlow as a fundamental optimizer and is utilized in a wide range of engineering fields, most notably in training neural networks; see, e.g., \cite{bertsekas2011,bottou2012,gurbu2019,sun2020optimization}.

\textbf{Motivations.} Despite its widespread practical usage, the theoretical understanding of $\RR$ in the nonconvex setting has been mainly limited to in-expectation complexity bounds and almost sure asymptotic convergence results. Though these results provide insightful characterizations of the performance of $\RR$, they either apply to the average case or are of asymptotic nature, differing partly from the practice where one only runs the method once for a finite number of iterations. In this work, we aim to establish high probability guarantees for $\RR$, providing both an ergodic complexity bound for finding an approximate stationary point and a simple stopping criterion for adaptively terminating the method.

\subsection{Our Results}\label{sec:main contributions}

Throughout this paper, we work under \Cref{Assumption:L smooth}, which is standard for analyzing $\RR$. We summarize our main results and core techniques below.

\textbf{High probability complexity guarantee.} We establish that, \emph{with high probability}, $\RR$ visits an $\varepsilon$-stationary point by achieving $\frac{1}{T}\sum_{t = 0}^{ {T-1}} \|\nabla f(x_t)\|^2\leq \varepsilon^2$ (without taking expectation) using at most $\tilde{\mathcal{O}} (\max \{ \sqrt{n}\varepsilon^{-3}, n \varepsilon ^{-2}\} )$ stochastic gradient evaluations (see \Cref{theo:constant-high prob complexity}). Here, the ``$\tilde{\mathcal{O}}$" hides logarithmic terms in the number of components $n$,  stationarity accuracy $1 /\varepsilon$, and probability parameter $1/\delta$. It is worth noting that our high probability sample complexity {matches} the best existing in-expectation complexity of $\RR$\footnote{Here, it refers to the in-expectation complexity for the original $\RR$. There are improved complexity results for different algorithmic oracles such as variance reduction method with $\RR$'s sampling scheme and $\RR$ with a specifically searched permutation order at each epoch; see, e.g., \cite{huang2021improved,malinovsky2021random,lu2022grab}.} \cite{mishchenko2020random,nguyen2020unified} up to a logarithmic term, under the same Lipschitz gradient assumption. Importantly, our result applies to every single run of $\RR$ with high probability, in contrast to the in-expectation results that average infinitely many runs. Our analysis does not impose any additional assumptions on the stochastic gradient errors nor any modifications to the $\RR$'s updating rule.

\textbf{Stopping criterion and last-iterate guarantee.} To further refine the above non-last iterate complexity result, we design a simple \emph{stopping criterion} embedded with a computable stopping test for $\RR$. This criterion terminates $\RR$ when the Euclidean norm of the accumulated stochastic gradients falls below a preset accuracy $\eta \varepsilon$, where $\eta > 0$ is some constant tolerance (e.g., $\eta =1$). $\RR$ equipped with such a stopping criterion is denoted as $\RRsc$, which introduces few additional computational loads compared to $\RR$. By designing a horizon-free blockwise step size rule, we prove that the stopping criterion must be triggered within $\tilde{\mathcal{O}} (\max \{ \sqrt{n}(\eta\varepsilon)^{-3}, n (\eta\varepsilon) ^{-2}\} )$ stochastic gradient evaluations with high probability (see \Cref{prop:stop RR dbl,thm:stop RR dbl}), aligning with the above sample complexity bound. Consequently, we establish a \emph{last-iterate} result which states that once $\RRsc$ is terminated by our stopping criterion, the returned iterate $x_\tau$ satisfies $\|\nabla f(x_\tau)\|\leq \mathcal O(\eta\varepsilon)$ (see \Cref{theo:last iterate dbl}). By contrast, verifying a stopping criterion could be intractable in the traditional in-expectation analysis of $\RR$ as the convergence bound is characterized in expectation, which is in general not computable in practice.

The key to deriving the aforementioned results is a new \emph{concentration property for random reshuffling} (see \Cref{sec:concentration}). This development properly exploits the randomness of the without-replacement sampling scheme of $\RR$ and yields concentration bounds with benign dependence on $n$ for the stochastic gradient errors, without using any additional assumptions beyond those used in $\RR$'s traditional in-expectation analysis nor modifying $\RR$'s original updating rule. It then allows us to provide a single-run analysis by conditioning on the randomness of $\RR$, forming the foundation for deriving our high probability results.
We believe that this technical development could provide insight on analyzing other stochastic optimization methods that involve sampling without replacement.

\subsection{Prior Arts}\label{sec:related works}

Thanks to its wide applications in large-scale optimization problems such as training neural networks, $\RR$ has gained significant attention recently for understanding its theoretical properties. Below, we present an overview of these theoretical findings, which is necessarily not exhaustive due to the extensive body of research on this topic.

\textbf{Complexity guarantees in expectation.} Unlike SGD that uses unbiased stochastic gradients, one of the main challenges in analyzing $\RR$ lies in the dependence between the stochastic gradients at each iteration. Various works have focused on deriving complexity bounds for $\RR$; see, e.g., \cite{haochen2019random,mishchenko2020random,nguyen2020unified,safran2020good,rajput2020closing, cha2023tighter}. For instance, the work \cite{mishchenko2020random} establishes an $\mathcal{O}(\sqrt{n} /\varepsilon)$ sample complexity for driving the expected squared distance between the iterate and the optimal solution below $\varepsilon$, under the assumptions that the objective function $f$ is strongly convex and each $f_i$ has Lipschitz continuous gradient. The authors concluded that $\RR$ outperforms SGD in this setting when $\varepsilon ^{-1} $ is relatively large based on this complexity result. In the smooth nonconvex case where $f$ is nonconvex and each $f_i$ has Lipschitz continuous gradient, it was shown in \cite{mishchenko2020random,nguyen2020unified} that $\RR$ has a sample complexity of $\mathcal{O}(\max\{\sqrt{n} \varepsilon ^{-3}, n\varepsilon^{-2} \} )$ for driving the expected Euclidean norm of the gradient below $\varepsilon$. However, all the mentioned complexity results for $\RR$ hold in expectation, characterizing the performance of the algorithm by averaging infinitely many runs. Hence, they may not effectively explain the performance of a single run of $\RR$. 

\textbf{Asymptotic convergence.} For strongly convex objective function with component Hessian being Lipschitz continuous, the work \cite{gurbu2019} presents that the squared distance between the q-suffix averaged iterate and the optimal solution converges to $0$ at a rate of $\mathcal{O}({1}/{t^2})$ with high probability, given that the sequence of iterates generated by $\RR$ is uniformly bounded. In the smooth nonconvex case, the almost sure asymptotic convergence result for the gradient norm was derived in \cite{li2022unified}. Additionally, the work \cite{li2023convergence} proves the almost sure asymptotic convergence rate results for $\RR$ under the Kurdyka-{\L}ojasiewicz inequality. The almost sure convergence result was further extended to a proximal-version of $\RR$ in \cite{qiu2025new}.  Though these asymptotic convergence results provide valuable theoretical guarantees, they primarily offer insights into the long-term/asymptotic behavior of the algorithm.

\textbf{High probability complexity guarantees.} Recently, there has been growing interest in studying high probability convergence behaviors for stochastic optimization methods, with a primary focus on SGD-type methods. The works \cite{ghadimi2013stochastic}, \cite{harvey2018tight}, and \cite{li2020high} obtain high probability complexity bounds for smooth nonconvex SGD, nonsmooth strongly convex SGD, and nonconvex adaptive SGD with momentum, respectively, all under the sub-Gaussian tailed stochastic gradient errors assumption.
However, such sub-Gaussian tail-type assumptions may be too optimistic in practice \cite{zhang2020adaptive}. Recent studies rely on the more practical standard bounded variance assumption for stochastic gradient errors. For smooth convex problems, \cite{eduard2020heavytail} analyzes clipped-SGD with momentum or a large batch size. For smooth nonconvex problems, \cite{cutkosky2021high} studies clipped-SGD with momentum and normalization. The work \cite{hubler2025gradient} further relaxes the need of clipping operation to normalization. One can observe that these analyses either impose the stringent sub-Gaussian tail-type assumptions on the stochastic gradient errors or require modifications to the algorithms. For $\RR$, the work \cite{lu2022general} treats it within a broader example-selection framework and verify an abstract average gradient error condition under a pointwise bounded gradient error assumption; by contrast, we directly exploit the randomness of the without-replacement sampling scheme for the original RR method itself to derive a benign concentration property without extra bounded error assumptions.

\textbf{Stopping criterion for the last iterate.} There are stopping criteria for nonconvex SGD-type methods; see, e.g., \cite{yin1990stopping,patel2022stopping,hunter2019automate} and the references therein. These proposals are either about discussing statistical stationarity or suggesting an asymptotic gradient-based stopping criterion. For $\RR$, the work \cite{liu2024last} provides last-iterate results in the convex setting. However, to our knowledge, a stopping criterion for $\RR$ in the nonconvex setting has not been studied. 

In summary, the existing literature still leaves two gaps for $\RR$ in smooth nonconvex finite-sum optimization. First, current complexity guarantees are either in-expectation or do not provide a direct single-run high probability theory for the original $\RR$ method that explicitly exploits the concentration property created by its without-replacement sampling scheme. Second, a simple adaptive stopping rule for $\RR$ with a non-asymptotic guarantee on the returned last iterate is still missing. In this paper, we address both gaps by establishing a high probability complexity guarantee for the original $\RR$ method through a direct analysis of its sampling-induced benign concentration, and by developing a simple stopping criterion that leads to a non-asymptotic last-iterate guarantee.

\section{Concentration Property for Random Reshuffling}\label{sec:concentration}
In this section, we study the concentration property of the stochastic gradient errors of $\RR$, which serves as the foundation for establishing our high probability complexity results.

We first provide a preparatory result below, which is a without-replacement matrix Bernstein's inequality that refines the result \cite{gross2010note}. We state it and prove it here in the exact form used in our later analysis. In particular, it features an intrinsic-dimension prefactor $\tilde d$ instead of the ambient dimension $d$, and hence yields better constants in our $\RR$ application.

\begin{lemma}[without-replacement matrix Bernstein's inequality]\label{lemma:concentration}
Let the set $\left\{X_1, \dots , X_{n} \right\}$ be a finite set of symmetric matrices. Suppose that the set is centered (i.e., $\overline{X} = \sum_{i=1}^{n} X_i /n=0$) and has a uniformly bounded operator $\ell_2$-norm $\| X_i \|_{\op} \le b$, $\forall i$. Suppose further that the permutation $\pi$ is sampled uniformly at random from $\Pi$ defined in \eqref{eq:permutations}. For any $1\le m \le n$ and $s\geq \sqrt{\lambda m /n} + b /3$, we have
\begin{equation}
\label{eq:without replacement sampling}
\mathbb{P} \left[ \lambda_{\max} \left( {\sum}_{i=1}^{m} X_{\pi^{i} }  \right) \ge s\right] \le 4 \tilde{d} \exp \left(- \frac{s^2 /2}{\lambda m/n+ bs /3}\right).
\end{equation}
Here, $\lambda m/n = \|V\|_{\op}$, i.e., the largest eigenvalue of the matrix $V = \frac{m}{n}\sum_{i=1}^{n} X_i^2$ and $\tilde{d} = {\operatorname{tr}(V)}/{\| V\|_{\op}} \geq 1$ is the intrinsic dimension of $V$.
\end{lemma}

\begin{proof}
Let
\[
Y_\pi:=\sum_{i=1}^m X_{\pi^i},
\qquad
Y_\sigma:=\sum_{i=1}^m X_{\sigma^i},
\]
where $\sigma^1,\ldots,\sigma^m$ are sampled i.i.d.\ and uniformly from $\{1,\ldots,n\}$ with replacement. Since $\overline X=0$ and $\|X_i\|_{\op}\le b$ for all $i$, the summands $\{X_{\sigma^i}\}_{i=1}^m$ are independent, centered, symmetric random matrices satisfying
\[
\left\|X_{\sigma^i}\right\|_{\op}\le b,
\qquad
\sum_{i=1}^m \mathbb E[X_{\sigma^i}^2]
=
\frac{m}{n}\sum_{i=1}^n X_i^2
=:V .
\]
Hence, $\lambda m/n=\|V\|_{\op}$ and $\tilde d=\operatorname{tr}(V)/\|V\|_{\op}$ is the intrinsic dimension of $V$. By the matrix Laplace transform \cite[Proposition 7.4.1]{tropp2015matrixconcentration}, for any $\theta>0$,
\[
\mathbb P\!\left[\lambda_{\max}(Y_\sigma)\ge s\right]
\leq \frac{1}{\exp (\theta s) - \theta s -1} \mathbb{E}_{\sigma } \left[\operatorname{tr} \left(e^{\theta Y_\sigma}-I\right)\right],
\]
where $\operatorname{tr} (e^{ \theta M} - I)$ is the corresponding trace exponential moment. As derived in \cite[Theorem 7.7.1]{tropp2015matrixconcentration}, optimizing the right-hand side over $\theta>0$, we obtain
\[
\mathbb P\!\left[\lambda_{\max}\left(Y_\sigma\right)\ge s\right]
\le
4\tilde d \exp\!\left(-\frac{s^2/2}{\lambda m/n+bs/3}\right),
\quad \text{when} \quad 
s\ge \sqrt{\lambda m/n}+b/3.
\]

It remains to transfer this bound from the i.i.d.\ sample $Y_\sigma$ to the without-replacement sample $Y_\pi$. By the matrix Laplace transform method again and the fact $\mathbb E_\pi[Y_\pi]=0$, for any $\theta>0$ and $s>0$,
\[
\mathbb P\!\left[\lambda_{\max}(Y_\pi)\ge s\right]
\le
\frac{1}{e^{\theta s}-\theta s-1}\,
\mathbb E_\pi\!\left[\operatorname{tr}\!\left(e^{\theta Y_\pi}-\theta Y_\pi-I\right)\right] = \frac{1}{e^{\theta s}-\theta s-1}\,
\mathbb E_\pi\!\left[\operatorname{tr}\!\left(e^{\theta Y_\pi}-I\right)\right].
\]
For any $\theta\in\mathbb R$, the mapping
$
M\mapsto \operatorname{tr}(e^{\theta M}-I)
$
is convex on the space of symmetric matrices. Therefore, Hoeffding's comparison principle for sampling without replacement \cite[Theorem 4]{hoeffding1963ineq} and \cite{gross2010note} yields
\[
\mathbb E_\pi\!\left[\operatorname{tr}(e^{ \theta Y_\pi}-I)\right]
\le
\mathbb E_\sigma\!\left[\operatorname{tr}(e^{ \theta Y_\sigma}-I)\right].
\]
Combining the above two inequalities provides 
\begin{equation}\label{eq:wr tmp}
\mathbb P\!\left[\lambda_{\max}(Y_\pi)\ge s\right]
\le
\frac{1}{e^{\theta s}-\theta s-1}
\mathbb E_\sigma\!\left[\operatorname{tr}(e^{\theta Y_\sigma}-I)\right].
\end{equation}
Since \eqref{eq:wr tmp} holds for every $\theta>0$, optimizing its right-hand side over $\theta>0$ in the same way as for $Y_\sigma$ yields exactly \eqref{eq:without replacement sampling}. 
\end{proof}

\textbf{Concentration property of the stochastic gradient errors in random reshuffling.} We now provide one of the main results of this work, i.e., bound the stochastic gradient errors in $\RR$. Towards that end, let us introduce two important quantities associated with the $t$-th epoch of $\RR$: 1) the accumulation of the stochastic gradients $g_t$, and 2) the gradient error $e_t$ caused by using $g_t$ to approximate the true gradient $\nabla f(x_t)$. They are defined as
\begin{equation}\label{eq:gt and et}
    \left[
    \begin{aligned}
        &g_t= \frac{1}{n}{\sum}_{i=1}^{n} \nabla f_{\pi_t ^{i}}(x_t^{i-1}),\\
        &e_t = g_t- \nabla f(x_t).
    \end{aligned}
    \right.
\end{equation}

The following theorem is a critical step for establishing our high probability complexity results. For the original $\RR$ method and under the standard variance-type control \eqref{eq:ABC condition} implied by Lipschitz gradients, it explicitly exploits the randomness of the without-replacement sampling scheme in $\RR$ and derive a concentration bound with a benign dependence on $n$ for $\RR$'s cumulative stochastic gradient errors. This is the key to deriving a sample complexity bound that has a better dependence on the number of component functions/data points $n$ compared to the deterministic incremental gradient method.

\begin{theorem}[concentration property of stochastic gradient errors]\label{thm:RR sampling gradients}
Let $\pi$ be sampled uniformly at random from $\Pi$ defined in \eqref{eq:permutations}.  Suppose that \Cref{Assumption:L smooth} is valid. Then, for any $x\in \mathbb{R}^d$ and $1\le m\le n$, the following inequality holds with probability at least $1-\delta$:
\begin{equation}\label{eq:RR sampling gradients}
   \left\| {\sum}_{i=1}^{m}\left(\nabla f_{\pi^{i} } (x)- \nabla f(x)\right) \right\|^2 \le 4n \left(\sA\left(f(x)- \bar f \right)+\sB\right)\log^2 ({8}/{\delta}).
\end{equation}
\end{theorem}

\begin{proof}
Fix any $x\in \mathbb{R}^d$, we define
\[
u_i := \nabla f_i(x)-\nabla f(x),\quad i=1,\ldots,n.
\]
Then,
it follows 
$
\sum_{i=1}^n u_i
=
\sum_{i=1}^n \nabla f_i(x)-n\nabla f(x)
=0.
$
For each $i\in [n]$, we construct
\begin{equation}\label{eq:construct Xi}
X_i :=
\begin{bmatrix}
0_{d\times d} & u_i\\
u_i^\top & 0
\end{bmatrix}
\in \mathbb{R}^{(d+1)\times (d+1)}.
\end{equation}
We also define
\[
\gamma := \sum_{i=1}^n \|u_i\|^2.
\]
If $\gamma=0$, then $u_i=0$ for all $i\in[n]$, and hence
$
\sum_{i=1}^m \bigl(\nabla f_{\pi^i}(x)-\nabla f(x)\bigr)=0.
$
Therefore, the desired inequality is trivially true. In the rest of the proof, we assume that $\gamma>0$.

For the constructed $X_i$ in \eqref{eq:construct Xi}, let us verify the conditions in \cref{lemma:concentration} and specify the values for $b$, $\lambda$, and $\tilde d$. Immediately, we have that $X_i$ is symmetric and satisfies 
\begin{equation}\label{eq:centered}
\overline X
=
\frac{1}{n}\sum_{i=1}^n X_i
=
\begin{bmatrix}
0_{d\times d} & \frac{1}{n}\sum_{i=1}^n u_i\\[1mm]
\left(\frac{1}{n}\sum_{i=1}^n u_i\right)^\top & 0
\end{bmatrix}
=0.
\end{equation}
Moreover, if $u_i=0$, then $X_i=0$. If $u_i\neq 0$, then we can compute
\begin{equation}\label{eq:eigenvalue calculation}
X_i
\begin{bmatrix}
u_i/\|u_i\|\\
1
\end{bmatrix}
=
\|u_i\|
\begin{bmatrix}
u_i/\|u_i\|\\
1
\end{bmatrix},
\quad
X_i
\begin{bmatrix}
-u_i/\|u_i\|\\
1
\end{bmatrix}
=
-\|u_i\|
\begin{bmatrix}
-u_i/\|u_i\|\\
1
\end{bmatrix}.
\end{equation}
It is immediate to verify that $\operatorname{rank}(X_i)\le 2$. Hence, when $u_i\neq 0$, the matrix $X_i$ has exactly two nonzero eigenvalues, namely $\|u_i\|$ and $-\|u_i\|$, and all remaining eigenvalues are zero. Therefore, we obtain
\begin{equation}\label{eq:variance bound rv}
\|X_i\|_{\op}
=
\|u_i\|
\le
\sqrt{{\sum}_{j=1}^n \|u_j\|^2}
\le
\sqrt{n\Bigl(\sA\bigl(f(x)-\bar f\bigr)+\sB\Bigr)}
=: b.
\end{equation}
where the last inequality is by \eqref{eq:ABC condition}. Next, for each $i\in[n]$, we have
$
X_i^2
=
\begin{bmatrix}
u_i u_i^\top & 0\\
0 & \|u_i\|^2
\end{bmatrix}
$, and hence 
\[
V = \frac{m}{n}\sum_{i=1}^n X_i^2
=
\frac{m}{n}\begin{bmatrix}
\sum_{i=1}^n u_i u_i^\top & 0\\
0 & \gamma
\end{bmatrix}.
\]
By the definition of $\lambda$ in \Cref{lemma:concentration}, we have $\lambda = \lambda_{\max}(\sum_{i=1}^n X_i^2)$. Note that $\sum_{i=1}^n X_i^2$ is a block diagonal matrix, and hence its eigenvalues are the union of the eigenvalues of the block $\sum_{i=1}^n u_i u_i^\top$ and the scaler $\gamma = \sum_{i=1}^n \|u_i\|^2$; see, e.g., \cite{horn2012matrix}. The fact that $
\lambda_{\max}\!\left(\sum_{i=1}^n u_i u_i^\top\right)
\le \operatorname{tr}\!\left(\sum_{i=1}^n u_i u_i^\top\right)
= \sum_{i=1}^n \|u_i\|^2 = \gamma$ shows that the largest eigenvalue of $\sum_{i=1}^n X_i^2$ is exactly $\gamma$. Therefore, we have
\begin{equation}\label{eq:op of V}
\lambda = \gamma \le n\Bigl(\sA\bigl(f(x)-\bar f\bigr)+\sB\Bigr)=b^2.
\end{equation}
where we have used \eqref{eq:ABC condition} in the inequality. 
The above calculation also directly gives 
$\|V\|_{\op}=\frac{m}{n}\lambda=\frac{m}{n}\gamma$ and 
$\operatorname{tr}(V)=\frac{2m}{n}\gamma$.
Since $\gamma>0$, the intrinsic dimension of $V$ is
\begin{equation}\label{eq:intrinsic dim}
\tilde d := \frac{\operatorname{tr}(V)}{\|V\|_{\op}}=2.
\end{equation}
Hence, combining \eqref{eq:centered}, \eqref{eq:variance bound rv}, \eqref{eq:op of V}, and \eqref{eq:intrinsic dim}, we have verified that the matrices $X_i$ in \eqref{eq:construct Xi} satisfy all conditions required in \Cref{lemma:concentration} with
\begin{equation}\label{eq:values for lemma1}
b=\sqrt{n\Bigl(\sA\bigl(f(x)-\bar f\bigr)+\sB\Bigr)},
\quad
\lambda\le b^2,
\quad
\tilde d=2.
\end{equation}

Now, let
\[
\ell_\delta := \log(8/\delta),
\qquad
s := 2b\ell_\delta.
\]
Since $\delta\le 1$, we have $\ell_\delta\ge \log 8 >1$. Using $\lambda \leq b^2$ in \eqref{eq:values for lemma1} and the fact that $m\le n$, we get
\[
\sqrt{\lambda m/n} + b/3 <2b \le 2b\ell_\delta = s.
\]
Therefore, the condition on $s$ in \Cref{lemma:concentration} is satisfied, and thus
\[
\mathbb{P}\!\left[
\lambda_{\max}\!\left(\sum_{i=1}^m X_{\pi^i}\right)\ge s
\right]
\le
4\tilde d
\exp\!\left(
-\frac{s^2/2}{\lambda m/n + bs/3}
\right).
\]
Using again \eqref{eq:values for lemma1} and $m\le n$, we further obtain
\begin{align*}
4\tilde d
\exp\!\left(
-\frac{s^2/2}{\lambda m/n + bs/3}
\right)
&\le
8\exp\!\left(
-\frac{2b^2\ell_\delta^2}{b^2+\frac{2}{3}b^2\ell_\delta}
\right)\\
&=
8\exp\!\left(
-\frac{2\ell_\delta^2}{1+\frac{2}{3}\ell_\delta}
\right)\\
&\le
8e^{-\ell_\delta}
=
\delta,
\end{align*}
where the last inequality follows from $\ell_\delta> 1$. Consequently, it follows
\begin{equation}\label{eq:tail bound}
\mathbb{P}\!\left[
\lambda_{\max}\!\left({\sum}_{i=1}^m X_{\pi^i}\right)
\ge
2b\log(8/\delta)
\right]
\le
\delta.
\end{equation}

Finally, by noting 
$
{\sum}_{i=1}^m X_{\pi^i}
=
\begin{bmatrix}
0_{d\times d} & {\sum}_{i=1}^m u_{\pi^i}\\
\left({\sum}_{i=1}^m u_{\pi^i}\right)^\top & 0
\end{bmatrix}
$ 
and applying the same eigenvalue calculation as in \eqref{eq:eigenvalue calculation} with
$
u := \sum_{i=1}^m u_{\pi^i}
$,
we obtain
\[
\lambda_{\max}\!\left({\sum}_{i=1}^m X_{\pi^i}\right)
=
\left\|{\sum}_{i=1}^m u_{\pi^i}\right\|
=
\left\|{\sum}_{i=1}^m \bigl(\nabla f_{\pi^i}(x)-\nabla f(x)\bigr)\right\|.
\]
Substituting this identity and the definition of $b$ from \eqref{eq:values for lemma1} into \eqref{eq:tail bound} and then squaring both sides, we conclude that, with probability at least $1-\delta$,
\[
\left\| {\sum}_{i=1}^{m}\bigl(\nabla f_{\pi^{i}}(x)- \nabla f(x)\bigr) \right\|^2
\le
4n \Bigl(\sA\bigl(f(x)- \bar f \bigr)+\sB\Bigr)\log^2 (8/\delta).
\]
This completes the proof. 
\end{proof}

Based on \Cref{thm:RR sampling gradients}, we can then explicitly invoke the scheme of $\RR$ and establish the following high probability bound for $\RR$'s gradient error $e_t$ defined in \eqref{eq:gt and et}.

\begin{corollary}[concentration property for random reshuffling]\label{lemma:stochastic error}
  Suppose that \Cref{Assumption:L smooth} is valid and the step size $\alpha_t$ in $\RR$ satisfies
  \begin{equation}\label{eq:step size cond}
     4 \alpha_t n \sL \le 1.
  \end{equation}
  Then, with probability at least $1- \delta$, we have
  \begin{equation}\label{eq:concentrate et}
	 \| e_t\|^2 \leq 2 \alpha_t^2 n^2 \sL^2
      \| \nabla f(x_t) \|^2 + 32 \alpha_t^2 n \sL^2 \left(\sA\left(f(x_t)- \bar f \right) + \sB\right)\log ^2 ({8n}/{\delta}).
   \end{equation}
\end{corollary}

\begin{proof}
By the definition of $e_t$, we have
\begin{equation}
    \| e_t \|^2  = \left\| \frac{1}{n} {\sum}_{i=1}^{n} \nabla f_{\pi_t^{i}  } (x_t^{i-1})- \frac{1}{n} {\sum}_{i=1}^{n}  \nabla f_{\pi_t^{i}  } (x_t) \right\|^2.
\end{equation}
Let us define $\Delta_t = \sum_{i=1}^{n} \| \nabla f_{\pi _t^{i} } (x_t^{i-1})-\nabla f_{\pi _t^{i} } (x_t) \|^2$. By the smoothness of $f_i$ in \Cref{Assumption:L smooth}, we obtain
\begingroup
\allowdisplaybreaks
\begin{align}\label{eq:sum of norm error}
	\Delta _t & \le \sL^2 {\sum}_{i=1}^{n} \| x_t^{i-1} -x_t \|^2 = \alpha_{t}^2 \sL^2 {\sum}_{i=1}^{n} \left\| {\sum}_{j=1}^{i-1} \nabla f_{\pi _t^j } (x_t^{j-1}) \right\|^2 \nonumber\\
	&\le 2\alpha_{t}^2 \sL^2 {\sum}_{i=1}^{n} \left(\left\| {\sum}_{j=1}^{i-1} \nabla f_{\pi _t^{j} } (x_t^{j-1})- \nabla f_{\pi _t^{j} } (x_t) \right\|^2 + \left\| {\sum}_{j=1}^{i-1} \nabla f_{\pi _t^{j} } (x_t) \right\|^2\right) \\
	&\le  2\alpha_{t}^2 \sL^2 {\sum}_{i=1}^{n} \left((i-1){\sum}_{j=1}^{i-1} \left\| \nabla f_{\pi_t^{j} } (x_t^{j-1})-\nabla f_{\pi_t^{j}}(x_t) \right\|^2\right. \nonumber \\ &\quad \left. + 2 \left\| {\sum}_{j=1}^{i-1} \left(\nabla f_{\pi _t^{j} } (x_t)-\nabla f(x_t)\right)\right\|^2 + 2 \left(i-1\right)^2 \left\| \nabla f(x_t) \right\|^2\right). \nonumber
\end{align}
\endgroup
Note that we have used the convention that $\sum_{j=1}^{i-1} = 0$ when $i=1$. Based on \eqref{eq:sum of norm error}, we can compute
\begin{align*}
    \Delta_t &\leq \alpha_t^2n^2\sL^2 \Delta_t + 4\alpha_{t}^2 \sL^2 {\sum}_{i=1}^{n} \left\| {\sum}_{j=1}^{i-1} \left(\nabla f_{\pi _t^{j} } (x_t)-\nabla f(x_t)\right)\right\|^2 \nonumber  \\
    &\quad + \frac{4}{3} \alpha_t^2 n^3\mathsf{L}^2 \|\nabla f(x_t)\| ^2,
\end{align*}
where we have used $\sum_{j=1}^{i-1} \| \nabla f_{\pi_t^{j} } (x_t^{j-1})-\nabla f_{\pi_t^{j}}(x_t) \|^2\leq \Delta_t$, $\sum_{i=1}^n (i-1) \leq n^2/2$, and $\sum_{i=1}^n (i-1)^2 \leq n^3/3$. To further provide a bound for the above inequality, we apply \Cref{thm:RR sampling gradients} by scaling the probability parameter from $\delta$ to $\delta/n$ for each $i$, and then apply union bound for all $1\leq i\leq n$. This provides us with probability at least $1-\delta$
\begin{align}\label{eq:union bound for Delta t}
    \Delta_t &\leq \alpha_t^2n^2\sL^2 \Delta_t + 16\alpha_t^2n^2\sL^2 \left(\sA\left(f(x_t)-\bar f\right)+\sB\right)\log ^2 ( {8n}/{\delta}) \nonumber \\
    &\quad +  \frac{4}{3} \alpha_t^2n^3\mathsf{L}^2 \|\nabla f(x_t)\| ^2.
\end{align}
Combining the terms on $\Delta_t$ in \eqref{eq:union bound for Delta t}, dividing both sides by $(1-\alpha_t^2n^2\sL^2)$, and using \eqref{eq:step size cond} provides with probability at least $1-\delta$
\begin{align*}
    \Delta_t \leq 2\alpha_{t}^2 n^3 \sL^2 \left\| \nabla f(x_t) \right\|^2 + 32 \alpha_{t}^2n^2\sL^2\left(\sA\left(f(x_t)-\bar f\right)+\sB\right)\log ^2 ( {8n}/{\delta}),
\end{align*}
 {where we have rounded the coefficients to the nearest upper integers for ease of presentation.} Finally, recognizing $\|e_t\|^2 \leq \frac{1}{n} \Delta_{t}$ establishes \eqref{eq:concentrate et}.
\end{proof}

\section{High Probability Complexity Guarantee for Random Reshuffling}\label{sec:complexity bound}
Using the established concentration property, in this section we first establish a high probability approximate descent property for $\RR$, and then prove a high probability ergodic complexity guarantee.

\textbf{Handling Randomness.} At each epoch $t$, the iterate $x_t$ is measurable w.r.t. the past randomness, while the permutation $\pi_t$ is freshly sampled and independent of $x_t$. Hence, \Cref{lemma:stochastic error} applies conditionally given $x_t$, and we apply union bound over iterations.

\begin{lemma}[approximate descent property]\label{lemma:descent lemma}
  Under the setting of \Cref{lemma:stochastic error}, the following inequality holds with probability at least $1-\delta$:
  \begin{align}\label{eq:descent lemma}
   f(x_{t+1}) - \bar f  &\leq \left(1+ 32  \sL^2 \sA \log^2 \left( {8n}/{\delta}\right)  \alpha_t^3 n^2\right)\left(f(x_t) - \bar f\right) \nonumber \\
   &\quad - \frac{\alpha_t n}{8}\| \nabla f(x_t) \|^2- \frac{\alpha _t n}{2}\left\lVert g_t \right\rVert^2 + 32  \sL^2 \sB \log^2 ( {8n}/{\delta}) \alpha_t ^3 n^2.
\end{align}

\end{lemma}

\begin{proof}
We note that the smoothness condition in \Cref{Assumption:L smooth} implies the descent lemma; see, e.g., \cite[Lemma 1.2.3]{nesterov2003}. Then, we can compute
\begingroup
\allowdisplaybreaks
\begin{align}\label{eq:descent property}
	& f(x_{t+1})\le f(x_t)- \alpha_t \left\langle \nabla f(x_t), {\sum}_{i=1}^{n} \nabla f_{\pi_t^i}(x_t^{i-1}) \right\rangle + \frac{\alpha_t^2 \sL}{2}\left\| {\sum}_{i=1}^{n} \nabla f_{\pi_t^i} (x_t^{i-1}) \right\|^2 \nonumber\\
    &= f(x_t)-\alpha_t \langle \nabla f(x_t), n \nabla f(x_t)+n e_t\rangle+ \frac{\alpha_{t}^2 \sL}{2}\| n \nabla f(x_t)+n e_t \|^2 \\
    &\le f(x_t)- \alpha_t n(1 - \alpha_{t} n \sL)\| \nabla f(x_t) \|^2 + \alpha_{t}^2 n^2 \sL \| e_t \|^2 + \alpha_{t} n \langle \nabla f(x_t), -e_t\rangle \nonumber \\
	&\leq f(x_t) - \frac{3\alpha_t n}{4} \| \nabla f(x_t) \|^2 + \alpha_{t}^2 n^2 \sL\| e_t \|^2 + \frac{\alpha_{t} n}{2}\left(\| \nabla f(x_t) \|^2 +\| e_t \|^2 - \| g_t \|^2\right) \nonumber \\
	&\le f(x_t)- \frac{\alpha_{t} n}{4}\| \nabla f(x_t) \|^2 - \frac{\alpha_{t} n}{2} \|g_t\|^2 + \alpha_{t} n\| e_t \|^2,\nonumber
\end{align}
\endgroup
where the equality is due to the definitions in \eqref{eq:gt and et},  {the third line is from the inequality $\|a+b\|^2 \leq 2 \|a\|^2 + 2 \|b\|^2$,} the fourth line  follows from \eqref{eq:step size cond} and the fact that $\langle a, b\rangle  = \frac{1}{2}(\|a\|^2 + \|b\|^2 - \|a - b\|^2)$,  {and the last inequality is obtained by \eqref{eq:step size cond} and rounding the coefficient of its last term to the nearest upper integer for display purpose.}
Finally, by subtracting $\bar f$ on both sides of the above inequality,   plugging \Cref{lemma:stochastic error} for upper bounding $\|e_t\|^2$, and utilizing \eqref{eq:step size cond}, we obtain the result.
\end{proof}

The approximate descent property in \cref{lemma:descent lemma} has an additional coefficient in the function value term, which does not allow the typical telescoping trick directly. This additional term originates from the general variance-type bound \eqref{eq:ABC condition}. In the next lemma, we refine the above approximate descent property for $\RR$. 

\begin{lemma}[refined approximate descent property]\label{lemma:constant-refined descent lemma}
Suppose that \Cref{Assumption:L smooth} is valid and the step size $\alpha_t$ satisfies
\begin{equation}\label{eq:constant-step rr}
  \alpha_t = \alpha := \min \left\{ \frac{1}{4 n\sL}, \frac{1}{(\sC_1 n^2 T)^{1/3}}\right\},
\end{equation}
 {where $T$ is the total number of iterations/epochs satisfying $T\geq 1$.} Then, with probability at least $1-\delta$, it holds that for all $0 \le t \le T-1$,
\begin{equation}\label{eq:constant-refined descent ineq}
  f(x_{t+1}) \le f(x_t) - \frac{\alpha_t n}{8} \left\lVert \nabla f(x_t) \right\rVert^2 - \frac{\alpha_t n}{2} \left\lVert g_t \right\rVert^2 + \alpha_t^3n^2\sG.
\end{equation}
Here, $\sC_1 = 32\sL^2\sA\log^2 \left( {8nT}/{\delta}\right)\geq 0$ and $\sG = \sC_1 \sF + \sC_2 \geq 0 $ with $\sF = 3(f(x_0) - \bar f) + 3\sB/\sA \geq 0$ and $\sC_2 = 32\sL^2\sB\log^2 \left( {8nT}/{\delta}\right) \geq 0$.
\end{lemma}

\begin{proof}
Dividing the probability parameter $\delta$ by $T$ in \cref{lemma:descent lemma} and then applying union bound for $0 \le t \le T-1$, we obtain
\begin{equation}\label{eq:apply descent}
  f(x_{t+1}) - \bar f  \leq (1+ \alpha_t^3 n^2 \sC_1)\left(f(x_t) - \bar f\right) - \frac{\alpha_t n}{8}\| \nabla f(x_t) \|^2- \frac{\alpha _t n}{2}\left\lVert g_t \right\rVert^2 + \alpha_t ^3 n^2 \sC_2,
\end{equation}
which holds for all $0 \le t\le T-1$ with probability at least $1-\delta$. Our remaining discussion is conditioned on the event in \eqref{eq:apply descent}, and hence all arguments are deterministic. Unrolling the above recursion gives
\begin{align*}
   f(x_t) - \bar f &\leq \left\{{\prod}_{i=0}^{t-1}(1+ \alpha_i^3 n^2 \sC_1)\right\} (f(x_0) - \bar f)  \\
  &\quad + {\sum}_{j=0}^{t-2} \left\{{\prod}_{i=j+1}^{t-1}(1+ \alpha_i^3 n^2 \sC_1)\right\} \alpha_j^3 n^2 \sC_2 + \alpha_{t-1} ^3 n^2 \sC_2.
\end{align*}
 {In the above inequality, we have used the conventions that ${\prod}_{i=0}^{t-1} = 1$ when $t = 0$, $\sum_{j=0}^{t-2} = 0$ when $0\leq t \leq 1$, and $\alpha_{t-1} = 0$ when $t=0$.} By the choice of our step size in \eqref{eq:constant-step rr} and the fact that $t\leq T$, we have
\[
   {\prod}_{i=0}^{t-1}(1+ \alpha_i^3 n^2 \sC_1) = \exp\left({\sum}_{i = 0}^{t-1} \log(1+ \alpha_i^3 n^2 \sC_1) \right) \leq \exp\left({\sum}_{i = 0}^{t-1}\alpha_i^3 n^2 \sC_1\right) \leq 3.
\]
Therefore, combining the above two inequalities provides
\begin{equation}\label{eq:bounded FV}
   f(x_t) - \bar f \leq 3 (f(x_0) - \bar f) + 3 {\sum}_{j = 0}^{T-1} \alpha_j^3n^2\sC_2 \leq 3 (f(x_0) - \bar f) + 3\sB/\sA = \sF
\end{equation}
for all $0\leq t\leq T$.
Plugging this upper bound into \eqref{eq:apply descent} yields \eqref{eq:constant-refined descent ineq}.
\end{proof}

With the developed machineries, we are now ready to establish the high probability ergodic sample complexity of $\RR$ for finding an approximate stationary point of problem \eqref{eq:problem}.

\begin{theorem}[high probability ergodic complexity guarantee for $\RR$]\label{theo:constant-high prob complexity}
Under the setting of \Cref{lemma:constant-refined descent lemma}, with probability at least $1-\delta$, we have
\begin{equation}\label{eq:rate}
  \frac{1}{T}{\sum}_{t=0}^{T-1} \| \nabla f(x_t) \|^2 \le \max \left\{ \frac{45\sL\sF}{T}, \frac{35 \sL^{2/3} \sA^{1/3} \sF \log ^{2 /3} (8nT /\delta)}{n ^{1 /3} T ^{2 /3}}\right\},
\end{equation}
Consequently, to achieve ${\sum}_{t=0}^{T-1} \| \nabla f(x_t) \|^2 /T \leq \varepsilon ^2$, $\RR$ needs at most
\begin{equation}\label{eq:constant-complexity}
n T = \tilde{\mathcal{O}}\left( \max \left\{ \sqrt{n} \varepsilon ^{-3} , n \varepsilon ^{-2}\right\}  \right)
\end{equation}
stochastic gradient evaluations, where $\tilde{\mathcal{O}}$ hides an additional $\log(\sqrt{n}\varepsilon^{-3}/\delta)$.
\end{theorem}

\begin{proof}
Summing up \eqref{eq:constant-refined descent ineq} from $t =0$ to $T-1$ and rearranging terms provide
\[
 \frac{1}{T} {\sum}_{t=0}^{T-1} \| \nabla f(x_t) \|^2 \leq \frac{8(f(x_0)- \bar f)}{\alpha n T} + 8\alpha^2 n \sG.
\]
When $n \geq \frac{\sA}{2\sL} T \log ^2 \left( 8nT /\delta\right)$, the step size $\alpha = 1 /4n\sL$ according to \eqref{eq:constant-step rr} and we have
\begin{equation}
  \label{eq:step1bound}
   \frac{1}{T} {\sum}_{t=0}^{T-1} \| \nabla f(x_t) \|^2 \leq \frac{32\sL(f(x_0) - \bar f)}{T} + \frac{\sG}{2n\sL^2} \leq \frac{45\sL\sF}{T};
\end{equation}
otherwise, $\alpha = 1/(\sC_1 n^2 T)^{1/3}$ and we have
\begin{equation}
\label{eq:step2bound}
  \frac{1}{T} {\sum}_{t=0}^{T-1} \| \nabla f(x_t) \|^2 \leq \frac{35 \sL^{2/3} \sA^{1/3} \sF \log ^{2 /3} (8nT /\delta)}{n ^{1 /3} T ^{2 /3}}.
\end{equation}
Combining the above two complexities gives \eqref{eq:rate}.  Letting the right-hand side of \eqref{eq:rate} equal to $\varepsilon ^2$ yields our final complexity result \eqref{eq:constant-complexity}.
\end{proof}

Our high probability sample complexity result for $\RR$ in  matches the best existing in-expectation one  \cite{mishchenko2020random,nguyen2020unified} up to a logarithmic term, under the same Lipschitz continuity assumption on the component gradients (i.e., \Cref{Assumption:L smooth}). Nonetheless, our result is applicable to every single realization of $\RR$ with high probability, providing a more practical picture of its performance.

Based on the concentration property developed in \Cref{sec:concentration}, it is also possible to obtain high probability results for $\RR$ when utilized to minimize strongly convex and convex functions. As this work mainly focuses on nonconvex setting, we omit the discussions on convex cases.

\section{ Stopping Criterion and Last-Iterate Result}\label{sec:stopp-crit-last}
The formulation of a stopping criterion constitutes a crucial part of algorithm design. In deterministic optimization, designing such a criterion can be relatively straightforward. For instance, one can examine the gradient function in the gradient descent method. However, it becomes significantly more challenging to construct a similar measure in the stochastic optimization regime. In the case of $\RR$, computing the full gradient function for monitoring stationarity is not feasible. Therefore, it necessitates the development of an estimated stopping criterion for $\RR$.

The study of a stopping criterion for $\RR$ is motivated by two factors: 1) It offers an adaptive stopping scheme as opposed to running the algorithm for a fixed number of iterations, potentially saving on execution time. 2) It yields a last-iterate result, which is especially meaningful in nonconvex optimization. We note that our high probability ergodic complexity bound derived in the previous section applies to $\min _{0 \le t \le T-1} \left\lVert \nabla f(x_t) \right\rVert$ rather than the last iterate. This discrepancy introduces the risk of returning the last iterate without satisfying the complexity bound, as also illustrated in \cite[Appendix H]{li2022unified}. Therefore, in this section, utilizing the analysis in the last section, we design a stopping criterion for $\RR$ and establish a high probability complexity guarantee for the last iterate returned by this stopping rule.

\IncMargin{1em}
\begin{algorithm}[t]
	\caption{$\RRsc$: Random Reshuffling with Stopping Criterion} \label{alg:stopping criterion}
	\KwIn{constant tolerance $\eta$,  target accuracy $\varepsilon$\;}
        \kwInit{$x_0\in \mathbb{R}^{d}$, $t = 0$\;}
	\While{true}{
		Set $g_t = 0$\;
		Update the step size $\alpha_t$ according to a certain rule\;
		Sample $\pi_t = \{\pi_t^{1} , \dots, \pi_t^{n}\}$ uniformly at random from $\Pi$ defined in \eqref{eq:permutations}\;
		Set $x_t^0 = x_t$\;
		\For{$i = 1, \ldots, n$}
            {
            $x_t^{i} = x_t^{i-1} -\alpha_t \nabla f_{\pi_t^{i}} \left(x_t^{i-1} \right)$\tcc*[r]{update}
            $g_t= g_t+ \nabla f_{\pi _t^{i} } \left(x_t^{i-1} \right) /n $\;
            }
        \eIf(\tcc*[f]{stopping criterion}){$\left\lVert g_t \right\rVert \le \eta \varepsilon$}
           {
           Set $\tau = t$\;
           \Return $x_\tau$\;
           }
           {
           Set $x_{t+1}= x_t^{n}$\;
           }
        Set $t = t+1$\;
	}
\end{algorithm}
\DecMargin{1em}

\subsection{Random Reshuffling with Stopping Criterion}\label{sec:stop criterion}
Our primary observation from \Cref{lemma:constant-refined descent lemma} is that the accumulation of the stochastic gradients $g_t$ (defined in \eqref{eq:gt and et}) almost mirrors the role of the true gradient for descent. This motivates us to track $g_t$ and use it as a stopping criterion. Importantly, thanks to the concentration property established in \Cref{sec:concentration}, the approximate descent property of $\RR$ in \Cref{lemma:constant-refined descent lemma} holds with high probability, without requiring taking any expectation operations. Consequently, the term $g_t$ in \Cref{lemma:constant-refined descent lemma} is practically computable and its computation imposes negligible additional computational burden.

Based on the above observations, we design $\RR$ with stopping criterion (denoted as $\RRsc$) in \Cref{alg:stopping criterion}. In this algorithm, we calculate the accumulation of the stochastic gradients used in the update and store it in $g_t$. After each epoch, $\RRsc$ checks
\begin{equation}\label{eq:stopping criterion}
\left\lVert g_t \right\rVert \leq \eta \varepsilon,
\end{equation}
where  $\varepsilon$ is the desired accuracy and $\eta>0$ is some constant  (e.g., $\eta =1$). Once this criterion is triggered, we stop the algorithm and return the last iterate $x_\tau$. In this section, we establish that the stopping criterion is guaranteed to be triggered with high probability and then provide complexity guarantees for the returned last iterate $x_\tau$.

\textbf{Horizon-free step size rule.}
Since a stopping rule-based method does not know its stopping horizon in advance, the step size used for deriving \Cref{theo:constant-high prob complexity} is no longer suitable for $\RRsc$. Instead, we design a horizon-free blockwise step size schedule for $\RRsc$. For each block index $k=0,1,2,\ldots$, let
\begin{equation}\label{eq:block length and epoch count}
S_k := 2^k,
\qquad
T_k := \sum_{j=0}^{k-1} S_j = 2^k-1,
\end{equation}
with the convention $T_0=0$. Thus, block $k$ has length $S_k$ and consists of the epochs
\[
t \in \{T_k,\ldots,T_k+S_k-1\}.
\]
We allocate blockwise failure probability budgets by
\[
\delta_k := \frac{6\delta}{\pi^2 (k+1)^2},
\qquad k=0,1,2,\ldots,
\]
so that $\sum_{k=0}^\infty \delta_k = \delta$. Within block $k$, we apply \Cref{lemma:stochastic error} at each epoch with failure probability $\delta_k/S_k$. By a union bound, the resulting bound on the stochastic error then holds simultaneously for all epochs in block $k$ with probability at least $1-\delta_k$. This gives rise to the logarithmic factor
\begin{equation}
\label{def:ellk} 
\ell_k := \log\!\Big(\frac{8n}{\delta_k/S_k}\Big)
=
\log\!\Big(\frac{8nS_k}{\delta_k}\Big), \quad k=0,1,2,\ldots.
\end{equation}
Accordingly, we choose a constant step size within each block:
\begin{equation}\label{eq:step-size-sc-dbl}
\alpha_t
=
\alpha_{\scc}^{(k)}
:=
\min\left\{
\frac{1}{4 n\mathsf{L}},
\frac{\eta\varepsilon}{8\sqrt{n \mathsf{AF}}\,\mathsf{L}\,
\ell_k}
\right\},
\qquad
t\in\{T_k,\ldots,T_k+S_k-1\}.
\end{equation}

With this horizon-free step size rule, the following lemma reveals the strict descent property of $\RRsc$ before triggering the stopping criterion.

\begin{lemma}[strict descent property of $\RRsc$ before stopping]
\label{lemma:descent before stop dbl}
Suppose that \Cref{Assumption:L smooth} holds and that the step sizes are chosen according to \eqref{eq:step-size-sc-dbl}. Then, there exists an event $\mathcal{E}$ with
$
\mathbb{P}[\mathcal{E}] \ge 1-\delta
$
such that, conditioned on $\mathcal{E}$, we have
\begin{equation}\label{eq:strict descent}
f(x_{t+1})-f(x_t)< - \frac{\alpha_t n}{4}\eta^2\varepsilon^2,
\qquad t=0,1,\ldots,\tau-1.
\end{equation}

\end{lemma}

\begin{proof}
For each epoch $t\ge 0$, let $k(t)$ be the unique block index such that
\[
t\in\{T_{k(t)},\ldots,T_{k(t)}+S_{k(t)}-1\}.
\]

\textbf{Step 1: Uniform stochastic error control over all epochs.}
For each block $k$ and each epoch $t\in\{T_k,\ldots,T_k+S_k-1\}$, we apply \Cref{lemma:stochastic error} with probability parameter $\delta_k/S_k$. Then, with probability at least $1-\delta_k/S_k$, we have
\begin{equation}\label{eq:sc-dbl-et}
\|e_t\|^2
\le 2 \alpha_t^2 n^2 \mathsf{L}^2\|\nabla f(x_t)\|^2
+ 32 \alpha_t^2 n\mathsf{L}^2 \bigl(\mathsf{A}(f(x_t)- \bar f ) + \mathsf{B}\bigr)\ell_{k(t)}^{2}.
\end{equation}
Taking a union bound over the $S_k$ epochs in block $k$, we obtain an event $\mathcal{E}_k$ with
$
\mathbb{P}[\mathcal{E}_k]\ge 1-\delta_k
$
such that \eqref{eq:sc-dbl-et} holds for all epochs in block $k$.
Finally, we define the good event
\begin{equation}\label{eq:good event}
\mathcal{E}:=\bigcap_{k\ge 0}\mathcal{E}_k.
\end{equation}
Since $\sum_{k=0}^\infty \delta_k=\delta$, another union bound gives
\[
\mathbb{P}(\mathcal{E})\ge 1-\delta.
\]
Hence, conditioned on $\mathcal{E}$, \eqref{eq:sc-dbl-et} holds for every epoch $t\ge 0$.

\textbf{Step 2: Strict descent before stopping, conditioned on $\mathcal{E}$.}
Our remaining arguments are conditioned on the event $\mathcal{E}$, and hence are deterministic. In the following, we first prove the non-increasing property of \(f(x_t)\), then refine this property into descent by a strictly positive value at each epoch. We prove by induction on $t$ that
\[
P(t):\qquad f(x_t)\le f(x_0), \qquad t=0,1,\ldots,\tau.
\]
The base case $P(0): f(x_0)\le f(x_0)$ is trivial. Now fix any $t\le \tau-1$ and assume $P(t)$ holds.
Then the argument below shows that
$
f(x_{t+1})-f(x_t)< -\frac{\alpha_t n}{4}\eta^2\varepsilon^2,
$
and therefore $f(x_{t+1})\le f(x_t)\le f(x_0)$, i.e., $P(t+1)$.
Hence $P(t)$ holds for all $t\le \tau$, and the strict descent inequality holds for all $t\le \tau-1$.

The derivation in \eqref{eq:descent property} provides
\begin{align*}
f(x_{t+1})
&\le f(x_t) - \frac{\alpha_t n}{4}\|\nabla f(x_t)\|^2
      - \frac{\alpha_t n}{2}\|g_t\|^2
      + \alpha_t n\|e_t\|^2 \\
&\le f(x_t) - \frac{\alpha_t n}{2}\|g_t\|^2
   + 32 \alpha_t^3 n^2 \mathsf{L}^2
     \bigl(\mathsf{A}(f(x_t)-\bar f)+\mathsf{B}\bigr)\ell_{k(t)}^2 \\
&\le f(x_t) - \frac{\alpha_t n}{2}\|g_t\|^2
   + 32 \alpha_t^3 n^2 \mathsf{L}^2
     \bigl(\mathsf{A}(f(x_0)-\bar f)+\mathsf{B}\bigr)\ell_{k(t)}^2,
\end{align*}
where the last step uses the induction hypothesis. 
By the definition of $\mathsf{F}$ in \Cref{lemma:constant-refined descent lemma}, we have
$
\mathsf{A}(f(x_0)-\bar f)+\mathsf{B}=\frac{\mathsf{A}\mathsf{F}}{3}.
$
Therefore,
$
32 \alpha_t^3 n^2 \mathsf{L}^2
\bigl(\mathsf{A}(f(x_0)-\bar f)+\mathsf{B}\bigr)\ell_{k(t)}^2
=
\frac{32}{3}\alpha_t^3 n^2 \mathsf{L}^2 \mathsf{A}\mathsf{F}\,\ell_{k(t)}^2.
$
Using the second term in \eqref{eq:step-size-sc-dbl}, we obtain
\[
\frac{32}{3}\alpha_t^3 n^2 \mathsf{L}^2 \mathsf{A}\mathsf{F}\,\ell_{k(t)}^2
\le \frac{\alpha_t n}{6}\eta^2\varepsilon^2.
\]
Hence,
\[
f(x_{t+1})-f(x_t)
\le - \frac{\alpha_t n}{2}\|g_t\|^2 + \frac{\alpha_t n}{6}\eta^2\varepsilon^2.
\]
Since $t\le \tau-1$, the stopping criterion has not been triggered yet, and thus
\[
\|g_t\|>\eta\varepsilon.
\]
It follows that
\[
f(x_{t+1})-f(x_t)
< - \frac{\alpha_t n}{4}\eta^2\varepsilon^2.
\]
It shows $f(x_{t+1})\le f(x_t)\le f(x_0)$, which closes the induction and proves the lemma. 
\end{proof}

We next show that the stopping rule of $\RRsc$ must be triggered after finitely many epochs on the good event from \Cref{lemma:descent before stop dbl}.

\begin{proposition}[finite stopping time of $\RRsc$]
\label{prop:stop RR dbl}
Under the setting of \Cref{lemma:descent before stop dbl} and let $\mathcal E$ be the event in \Cref{lemma:descent before stop dbl}. For any block index $K\ge 0$, if it satisfies inequality
\begin{equation}\label{eq:stop-condition-K}
\sum_{k=0}^{K} \frac{nS_k\alpha_{\scc}^{(k)}}{4}\eta^2\varepsilon^2
\geq
f(x_0)-\bar f,
\end{equation}
then, on $\mathcal E$, the stopping time $\tau$ of $\RRsc$ satisfies
\begin{equation} \label{eq:stopping time}
\tau \le T_{K+1}-1.
\end{equation}
Here, $T_{K+1} := \sum_{k=0}^K S_k$ is the starting epoch of block $K+1$, and hence $\tau \leq T_{K+1} - 1$ means that $\RRsc$ stops by the end of block $K$. 
\end{proposition}

\begin{proof}
Condition on the event $\mathcal E$. Suppose, for contradiction, that
\[
\tau \ge T_{K+1}.
\]
Then, \Cref{lemma:descent before stop dbl} gives
\[
f(x_{t+1})-f(x_t)< -\frac{\alpha_t n}{4}\eta^2\varepsilon^2,
\qquad t=0,1,\ldots,T_{K+1}-1.
\]
Summing the above inequalities over $t=0,\ldots,T_{K+1}-1$, we obtain
\begin{align*}
f(x_{T_{K+1}})-f(x_0)
&<
-\sum_{t=0}^{T_{K+1}-1}\frac{\alpha_t n}{4}\eta^2\varepsilon^2 \\
&=
-\sum_{k=0}^{K}\sum_{t=T_k}^{T_k+S_k-1}\frac{\alpha_t n}{4}\eta^2\varepsilon^2 \\
&=
-\sum_{k=0}^{K}\frac{nS_k\alpha_{\scc}^{(k)}}{4}\eta^2\varepsilon^2 \\
&\leq
-(f(x_0)-\bar f),
\end{align*}
where in the third line we used that $\alpha_t=\alpha_{\scc}^{(k)}$ is constant in block $k$.
Therefore,
\[
f(x_{T_{K+1}}) < \bar f,
\]
which contradicts the fact that $\bar f$ is a lower bound of $f$. Hence $\tau \le T_{K+1}-1$. 
\end{proof}

We next quantify the stopping complexity of $\RRsc$ under the blockwise step size rule. By \Cref{prop:stop RR dbl}, it suffices to find a block index $K$ for which the condition \eqref{eq:stop-condition-K} holds. 
 
\begin{theorem}[stopping complexity of $\RRsc$]
\label{thm:stop RR dbl}
Under the setting of \Cref{lemma:descent before stop dbl}, with probability at least $1-\delta$, the number of stochastic gradient evaluations needed to trigger the stopping criterion satisfies
\[
n\tau
=
\tilde{\mathcal O}\!\left(
\max\left\{
\sqrt n\,(\eta\varepsilon)^{-3},
\;
n(\eta\varepsilon)^{-2}
\right\}
\right).
\]
\end{theorem}

\begin{proof}
Let $\mathcal E$ be the event in \Cref{lemma:descent before stop dbl}. Since
$\mathbb P(\mathcal E)\ge 1-\delta$, it suffices to bound $\tau$ on $\mathcal E$.
For each block $k\ge 0$, recall that
$
\ell_k:=\log\!\left(\frac{8nS_k}{\delta_k}\right).
$
Since $S_k=2^k$ and $\delta_k=6\delta/(\pi^2(k+1)^2)$, the sequence
$\{\ell_k\}_{k\ge 0}$ is increasing. Hence the blockwise step sizes (see \eqref{eq:step-size-sc-dbl})
\[
\alpha_{\scc}^{(k)}
=
\min\left\{
\frac{1}{4n\mathsf L},
\frac{\eta\varepsilon}{8\sqrt{n\mathsf A\mathsf F}\,\mathsf L\,\ell_k}
\right\}
\]
are nonincreasing in $k$. Therefore, the left side of condition \eqref{eq:stop-condition-K} can be further lower bounded as 
\begin{equation}
\begin{aligned}
\sum_{k=0}^{K} \frac{nS_k\alpha_{\scc}^{(k)}}{4}\eta^2\varepsilon^2
\ge \frac{n\alpha_{\scc}^{(K)}}{4}\eta^2\varepsilon^2 \sum_{k=0}^{K} S_k  &=
\frac{nT_{K+1}\alpha_{\scc}^{(K)}}{4}\eta^2\varepsilon^2\\
&=
\min\left\{
\frac{T_{K+1}}{16\mathsf L}\eta^2\varepsilon^2,\,
\frac{T_{K+1}\sqrt n}{32\sqrt{\mathsf A\mathsf F}\,\mathsf L\,\ell_K}\eta^3\varepsilon^3
\right\}.
\end{aligned}
\end{equation}
Thus, any block index $K$ satisfying
\begin{equation}\label{eq:bound epochs for stopping}
T_{K+1}
\geq
\max\left\{
16\mathsf L(f(x_0)-\bar f)(\eta\varepsilon)^{-2},\,
32\sqrt{\mathsf A\mathsf F}\,\mathsf L(f(x_0)-\bar f)\,
n^{-1/2}\ell_K(\eta\varepsilon)^{-3}
\right\}
\end{equation}
also satisfies condition \eqref{eq:stop-condition-K}. Define
\[
a:=16\mathsf L(f(x_0)-\bar f)(\eta\varepsilon)^{-2},
\qquad
b:=32\sqrt{\mathsf A\mathsf F}\,\mathsf L(f(x_0)-\bar f)\,n^{-1/2}(\eta\varepsilon)^{-3}.
\]
Let
\[
\bar K:=\min\left\{K\ge 0:\ T_{K+1}\ge \max\{a,\;b\ell_K\}\right\}.
\]
Then \eqref{eq:bound epochs for stopping} holds with \(K=\bar K\). Hence, by \Cref{prop:stop RR dbl}, on \(\mathcal E\),
$
\tau \le T_{\bar K+1}-1
$. It remains to bound \(T_{\bar K+1}\). If \(\bar K=0\), then \(T_{\bar K+1}=1\), and the result is immediate. Henceforth assume \(\bar K\ge 1\). By the minimality of \(\bar K\) and the monotonicity of \(\{\ell_k\}_{k\ge 0}\),
\[
T_{\bar K}
<
\max\{a,\;b\ell_{\bar K-1}\}
\le
\max\{a,\;b\ell_{\bar K}\}.
\]
Therefore, by \eqref{eq:block length and epoch count}
\[
T_{\bar K+1}
=
2T_{\bar K}+1
\le
2\max\{a,\;b\ell_{\bar K}\}+1
\le
2a+2b\ell_{\bar K}+1.
\]
Since
$
S_{\bar K}=\frac{T_{\bar K+1}+1}{2}$, and 
$\bar K+1=\log_2(T_{\bar K+1}+1)
$ by \eqref{eq:block length and epoch count}, 
we have
\[
\ell_{\bar K} =
\log\left( \frac{8nS_{\bar K}\pi^2(\bar K+1)^2}{6\delta} \right)
\le
\log \left(
\frac{2\pi^2 n (T_{\bar K+1}+1)\log_2^2(T_{\bar K+1}+1)}{3\delta}
\right).
\]
Hence,
\[
T_{\bar K+1}
\le
2a
+
2b\log\left(
\frac{2\pi^2 n (T_{\bar K+1}+1)\log_2^2(T_{\bar K+1}+1)}{3\delta}
\right)
+1.
\]
A standard inversion of the inequality above yields
\[
T_{\bar K+1}
=
\tilde{\mathcal O}(a+b)
=
\tilde{\mathcal O}\!\left(
\max\left\{
\mathsf L(f(x_0)-\bar f)(\eta\varepsilon)^{-2},\;
\sqrt{\mathsf A\mathsf F}\,\mathsf L(f(x_0)-\bar f)\,n^{-1/2}(\eta\varepsilon)^{-3}
\right\}
\right),
\]
where $\tilde{\mathcal O}$ hides logarithmic factors in $n$, $1/\delta$, and $1/\epsilon$.
Therefore, we finally conclude that
\[
n\tau
=
\tilde{\mathcal O}\!\left(
\max\left\{
n(\eta\varepsilon)^{-2},\;
\sqrt n\,(\eta\varepsilon)^{-3}
\right\}
\right)
\]
on \(\mathcal E\), and hence with probability at least \(1-\delta\). 
\end{proof}

\subsection{The Last-Iterate Result}\label{lastIterate}
In this subsection, we show that when $\RRsc$ terminates, the returned last iterate satisfies a constant-factor stationarity guarantee, namely,
\[
   \|\nabla f(x_\tau)\|\leq \mathcal O(\eta\varepsilon).
\]
The following lemma establishes the fact that small $\|g_t\|$ implies small $\|\nabla f(x_t)\|$.

\begin{lemma}\label{lemma:bound grad by gt dbl}
Under the setting of \Cref{lemma:descent before stop dbl}, with probability at least $1-\delta$,
\[
\left\lVert \nabla f(x_t) \right\rVert ^2 \le \frac{8}{3} \left\lVert g_t \right\rVert^2 + \frac{4}{9} \eta^2 \varepsilon ^2, \quad \forall t \leq \tau.
\]
\end{lemma}

\begin{proof}
Fix any $t\le \tau$ and let $k(t)$ be its block index. Since $\nabla f(x_t)=g_t-e_t$ by \eqref{eq:gt and et}, we have
\[
\|\nabla f(x_t)\|^2 \le 2\|g_t\|^2 + 2\|e_t\|^2.
\]
On $\mathcal{E}$ (i.e., with probability at least $1-\delta$), using \eqref{eq:sc-dbl-et} to bound $\|e_t\|$ gives
\begin{align}\label{eq:sc-dbl-grad-gt-raw}
\|\nabla f(x_t)\|^2
&\le 2\|g_t\|^2 + 4 \alpha_t^2 n^2 \sL^2\| \nabla f(x_t) \|^2
+ 64 \alpha_t^2 n\sL^2 \left(\sA(f(x_t)- \bar f ) + \sB\right)\ell_{k(t)}^{2}.
\end{align}
Moreover, by \Cref{lemma:descent before stop dbl} we have $f(x_t)\le f(x_0)$ for all $t\le \tau$, hence
$
\sA(f(x_t)-\bar f)+\sB \le \sA(f(x_0)-\bar f)+\sB=\frac{\sA\sF}{3}
$, where $\sF$ is defined in \Cref{lemma:constant-refined descent lemma}.
Substituting this into \eqref{eq:sc-dbl-grad-gt-raw} yields
\[
\|\nabla f(x_t)\|^2
\le 2\|g_t\|^2 + 4 \alpha_t^2 n^2 \sL^2\| \nabla f(x_t) \|^2
+ \frac{64}{3} \alpha_t^2 n\sL^2 \sA\sF\,\ell_{k(t)}^{2}.
\]
Rearranging and using $\alpha_t\le \frac{1}{4n\sL}$ (so that $1-4\alpha_t^2n^2\sL^2\ge 3/4$) gives
\begin{align*}
\|\nabla f(x_t)\|^2
&\le \frac{4}{3}\left(2\|g_t\|^2 + \frac{64}{3}\alpha_t^2 n\sL^2 \sA\sF\,\ell_{k(t)}^2\right)
= \frac{8}{3}\|g_t\|^2 + \frac{256}{9}\alpha_t^2 n\sL^2 \sA\sF\,\ell_{k(t)}^2.
\end{align*}
Finally, by the second term in \eqref{eq:step-size-sc-dbl},
$\alpha_t\le \frac{\eta\varepsilon}{8\sqrt{n\sA\sF}\,\sL\,\ell_{k(t)}}$ and thus
$\alpha_t^2 n\sL^2\sA\sF\,\ell_{k(t)}^2 \le \eta^2\varepsilon^2/64$.
Therefore,
\[
\|\nabla f(x_t)\|^2 \le \frac{8}{3}\|g_t\|^2 + \frac{256}{9}\cdot\frac{\eta^2\varepsilon^2}{64}
\le \frac{8}{3}\|g_t\|^2 + \frac{4}{9}\eta^2\varepsilon^2,
\]
which holds for all $t\le \tau$. 
\end{proof}

When $\RRsc$ stops at iteration $\tau$, we have $\|g_\tau\| \leq \eta \varepsilon$. In addition, the above lemma indicates when $\|g_t\|$ is small, the true gradient $\|\nabla f(x_t)\|$ can also be made small once the step size is appropriately chosen. This observation motivates us to derive the property of the true gradient when the method terminates, yielding a last-iterate complexity result.

\begin{theorem}[last-iterate guarantee]\label{theo:last iterate dbl}
Under the setting of \Cref{lemma:descent before stop dbl}, with probability at least $1-\delta$, $\RRsc$ is guaranteed to terminate at some iteration $\tau$ satisfying $n\tau =\tilde{\mathcal O}\!\left(
\max\left\{n(\eta\varepsilon)^{-2},\,\sqrt n\,(\eta\varepsilon)^{-3}\right\}\right)$, and the returned last iterate $x_\tau$ has $\|\nabla f(x_\tau) \| \le \sqrt{\frac{28}{9}}\eta\varepsilon$. 
\end{theorem}

\begin{proof}
On the event $\mathcal E$ in \Cref{lemma:descent before stop dbl}, finite termination and complexity of $\RRsc$ follows directly from \Cref{prop:stop RR dbl} and \Cref{thm:stop RR dbl}.
At the stopping time $\tau$, the stopping criterion gives $\|g_\tau\|\le \eta\varepsilon$. Plugging this into \Cref{lemma:bound grad by gt dbl} yields
\[
\|\nabla f(x_\tau)\|^2
\le \frac{8}{3}\|g_\tau\|^2 + \frac{4}{9}\eta^2\varepsilon^2
\le \frac{8}{3}\eta^2\varepsilon^2 + \frac{4}{9}\eta^2\varepsilon^2
=\frac{28}{9}\eta^2\varepsilon^2. 
\] 
\end{proof}
Here is a remark on this last-iterate guarantee. Suppose that the stopping criterion is triggered at iteration $t$. Our $\RRsc$ returns $x_t$ rather than $x_{t+1}$ after running the $(t+1)$-th epoch. Indeed, we can also return $x_{t+1}$. By the Lipschitz continuity of the gradient function, we have
\[
   \|\nabla f(x_{t+1})\| \leq \|\nabla f(x_t)\| + \|\nabla f(x_{t+1}) - \nabla f(x_t)\| \leq \|\nabla f(x_t)\| + \alpha_t n \sL \|g_t\| \leq \frac{9}{4} \eta\varepsilon.
\]
Thus, one could also return $x_{t+1}$ as $x_\tau$ without sacrificing the last-iterate guarantee.

\textbf{False positive and false negative control.}
The stopping test $\|g_t\|\le \eta \varepsilon$ can be viewed as a computable surrogate for the intractable gradient test, but up to two different constant factors. Define
\[
\bar c := \sqrt{\tfrac{28}{9}}\,\eta,
\qquad
\munderbar{c} := \sqrt{\tfrac{8}{27}}\eta.
\]
All statements below are on the same high probability event $\mathcal E$ (with $\mathbb P[\mathcal E]\ge 1-\delta$) used in \Cref{lemma:descent before stop dbl,lemma:bound grad by gt dbl,theo:last iterate dbl}. By \Cref{theo:last iterate dbl}, at the stopping time $\tau$ we have
$
\|\nabla f(x_\tau)\|\le \bar c\varepsilon.
$
Thus, the stopping rule has no false positives relative to the threshold $\bar c\varepsilon$: Whenever the test $\|g_t\|\le \eta\varepsilon$ is triggered, the returned iterate is guaranteed to satisfy
$
\|\nabla f(x_\tau)\|\le \bar c\varepsilon.
$

Our stopping rule also controls false negatives in a quantitative sense. Following essentially the same argument as in \Cref{lemma:bound grad by gt dbl}, on $\mathcal E$ we have
\begin{equation}\label{eq:false negative control}
   \|g_t\|^2 \le \frac{9}{4}\|\nabla f(x_t)\|^2 + \frac{1}{3}\eta^2\varepsilon^2,
   \qquad t\le \tau.
\end{equation}
Therefore,
\[
\|\nabla f(x_t)\|\le \munderbar{c}\varepsilon
\quad\Longrightarrow\quad
\|g_t\|\le \eta\varepsilon,
\]
so the stopping test must trigger once the true gradient norm drops below $\munderbar{c}\varepsilon$.

Hence, the computable criterion $\|g_t\|\le \eta\varepsilon$ is sound at level $\bar c\varepsilon$ and sensitive at level $\munderbar{c}\varepsilon$. Since $\bar c>\munderbar{c}$, these two levels are necessarily different. Accordingly, the surrogate stopping test nicely approximates the ideal gradient test up to explicit constant factors.

\section{Numerical Experiments}\label{sec:experiments}
In this section, we conduct practical classification experiments on the widely recognized MNIST dataset.  Our model of choice is a two-hidden layer fully connected neural network, which utilizes the smooth $\tanh$ activation function  {used in \cite{lecun1998gradient,glorot2010understanding}} and logistic regression in the final layer for the classification task. Each hidden layer in our network comprises 50 units. The training algorithms implemented are $\RR$ and SGD. We ensure fairness in comparison by using the same parameter settings for both algorithms. Specifically, the initial point is obtained by running the default initializer of PyTorch, which generates the initial weight matrices with entries following an i.i.d. uniform distribution. We use a batch size of 8 and an initial learning rate of 0.05, which is subsequently step-decayed by a factor of 0.7 after each epoch. We note that the focus of this work is on theory, and hence the experiments are only illustrative and use a common step-decay step size schedule rather than our theoretical requirements. This step-decay procedure follows the convention in the training of neural networks.  We conduct 100 independent trials for each algorithm to ensure a comprehensive evaluation. Our code for reproducing the experiment results is available at \url{https://github.com/hengxuyu/high_probability_guarantees_for_random_reshuffling}.

Our theoretical results rely on the Lipschitz gradient condition in \Cref{Assumption:L smooth}. It is in general hard to fully verify this condition for our neural network training problem. In this section, we conduct experiments to partly verify them along the trajectory of $\RR$ for this specific training problem.

For verifying the Lipschitz gradient condition, we note that the actual Lipschitz gradient condition we used in \Cref{lemma:stochastic error} can be verified if the estimate
\[
    \widehat {\sf L} = \frac{\|\nabla f_{\pi_i}(x_t^i) - \nabla f_{\pi_i}(x_t)\|}{\|x_t^i - x_t\|}
\]
is reasonably upper bounded for epoch $t$ and inner iteration $1\leq i \leq n$. 

The experiment results for the first epoch (i.e., $t=0$) are displayed in \Cref{fig:lip_grad}. One can observe that $\widehat \sL$ is uniformly bounded by some reasonable constant. Interestingly, one can observe that the estimates quickly decrease to small values when the inner iteration $i$ increases. We suspect that the algorithm quickly enters a benign local landscape with benign Lipschitz gradient. Further investigation of the benign local landscape behavior is left as future work. This experimental result partly justifies the Lipschitz gradient condition for this neural network training problem.

\begin{figure}[t]
    \centering
    \includegraphics[width=0.4\linewidth]{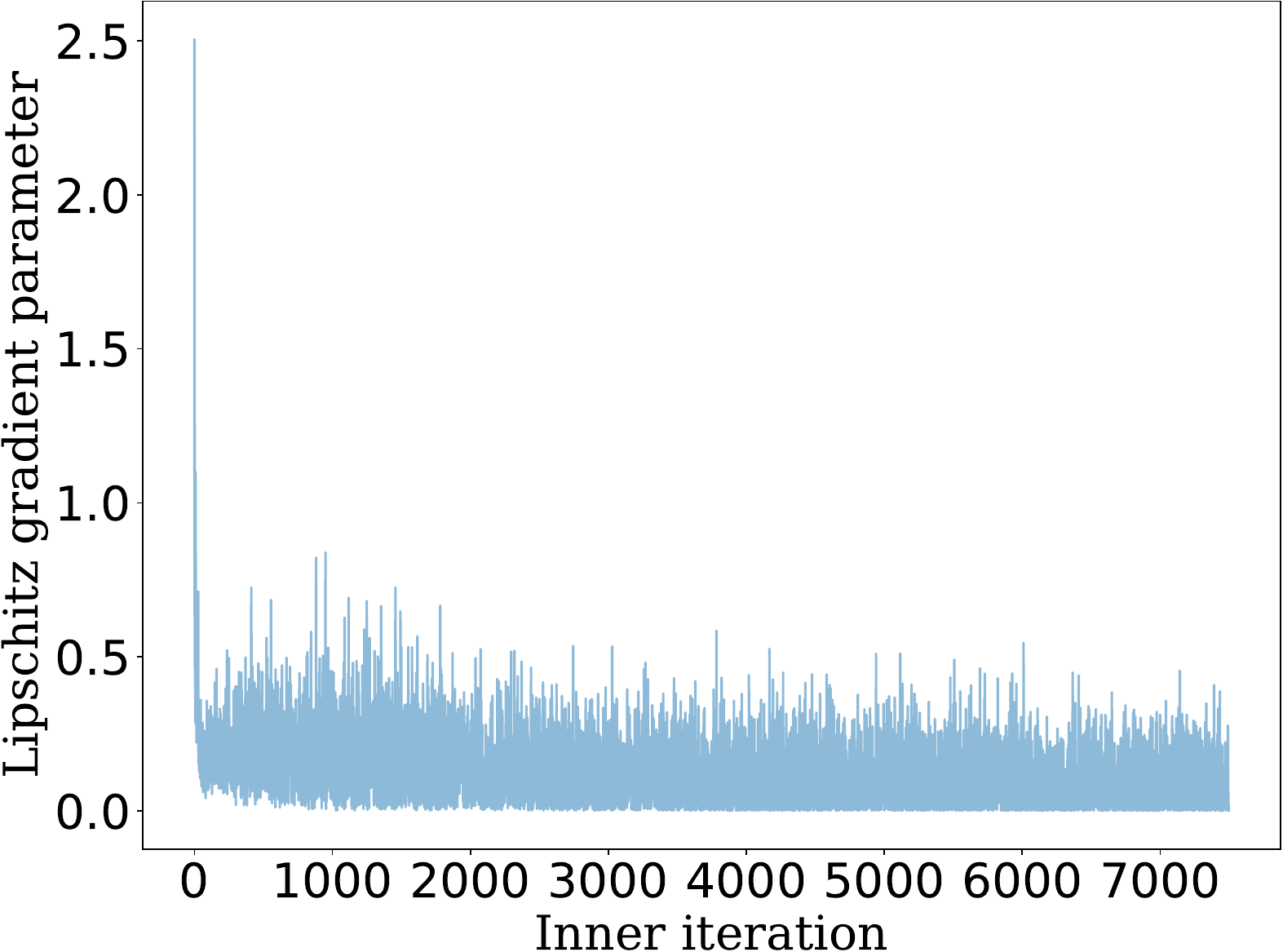}
    \caption{Verification of the Lipschitz gradient condition used in our theoretical developments.}\label{fig:lip_grad}
\end{figure}

\begin{figure}[t]
    \centering
    \begin{subfigure}[t]{0.48\linewidth}
    \includegraphics[width=0.8\linewidth]{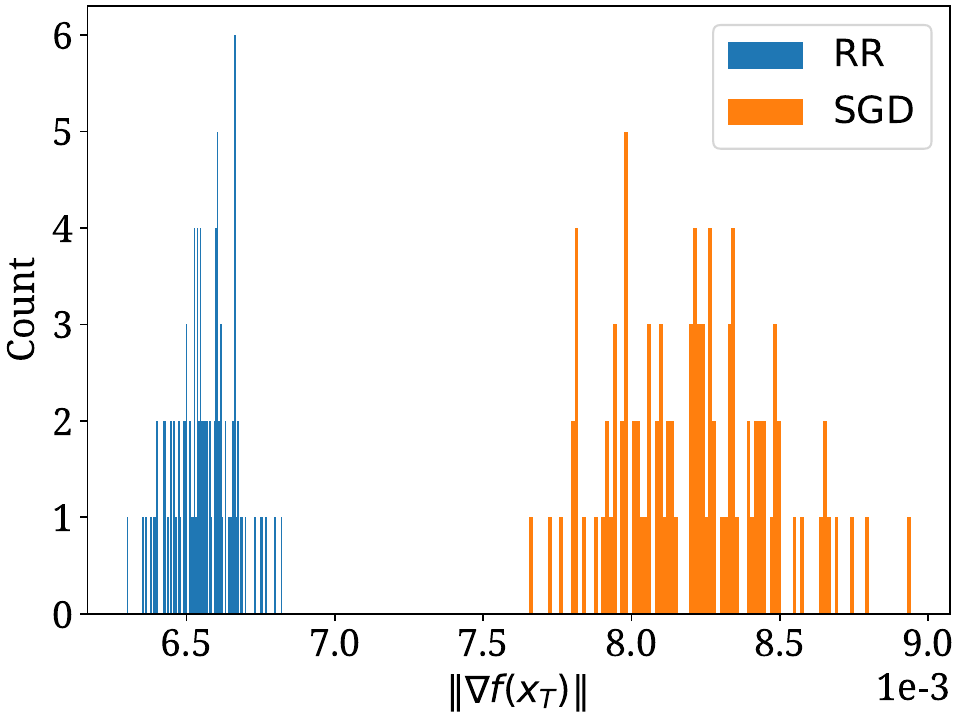}
	\caption{Gradient norm statistics for the last iterate over 100 independent runs.}
	\label{fig:lft}
    \end{subfigure}\hfill
    \begin{subfigure}[t]{0.48\linewidth}
    \includegraphics[width=0.85\linewidth]{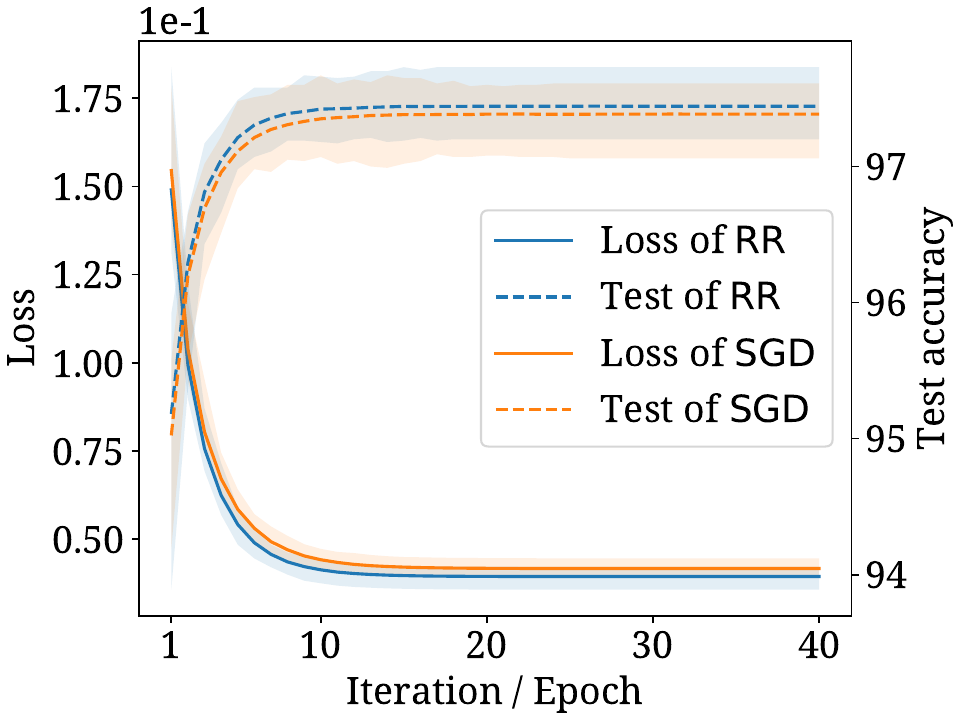}
	\caption{Training loss and test accuracy of $\RR$ and SGD.}
	\label{fig:losstest}
    \end{subfigure}
    \caption{Comparison of performance between $\RR$ and SGD.}
\end{figure}

\begin{figure}[t]
    \centering
    \begin{subfigure}{0.48\linewidth}
    \includegraphics[width=0.79\linewidth]{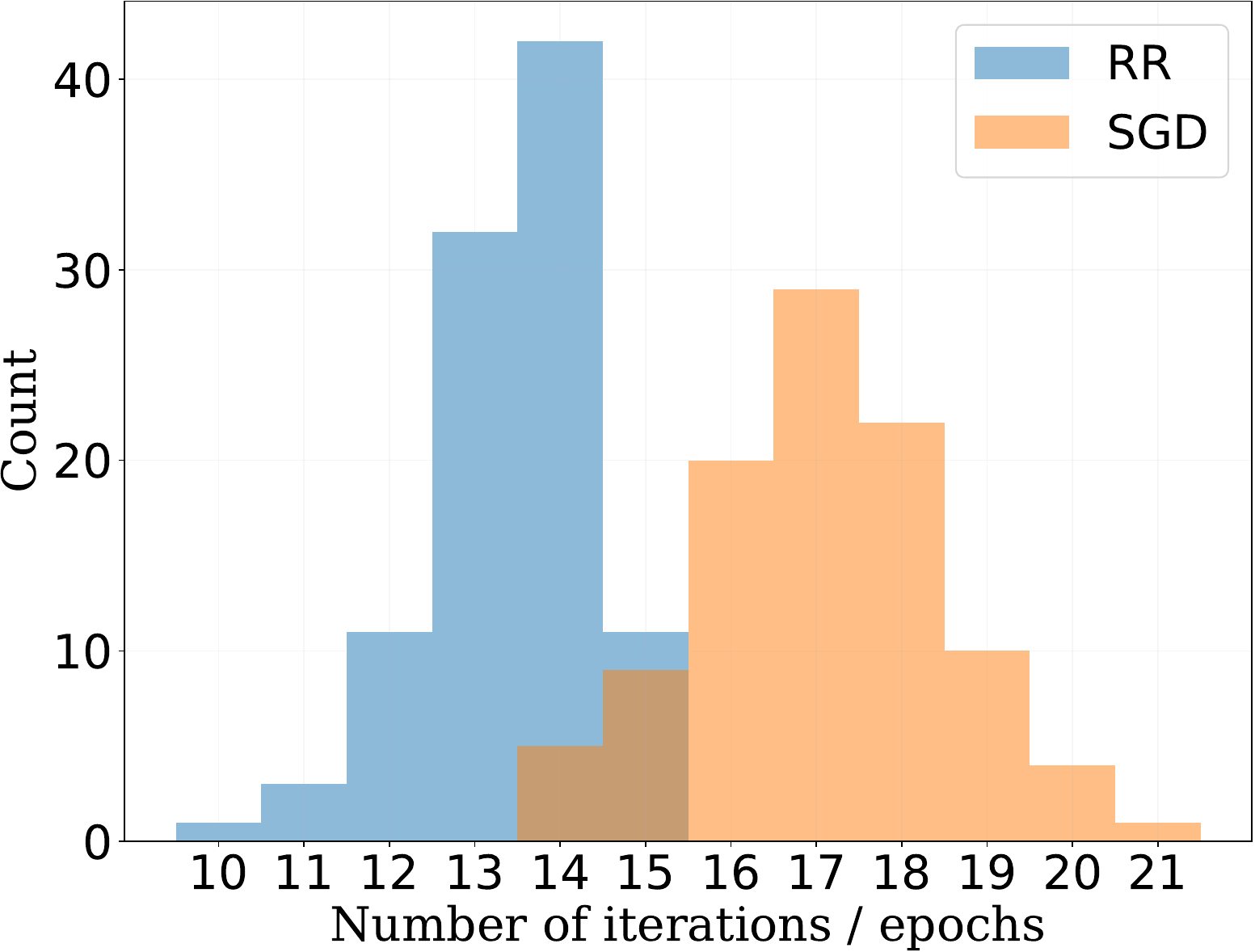}
	\caption{ {Histogram of iterations / epochs $t$ for achieving $\|\nabla f(x_t)\|\leq 10^{-2}$ over 100 independent runs.}}
	\label{fig:hist_1e-2}
    \end{subfigure}\hfill
    \begin{subfigure}{0.48\linewidth}
    \includegraphics[width=0.79\linewidth]{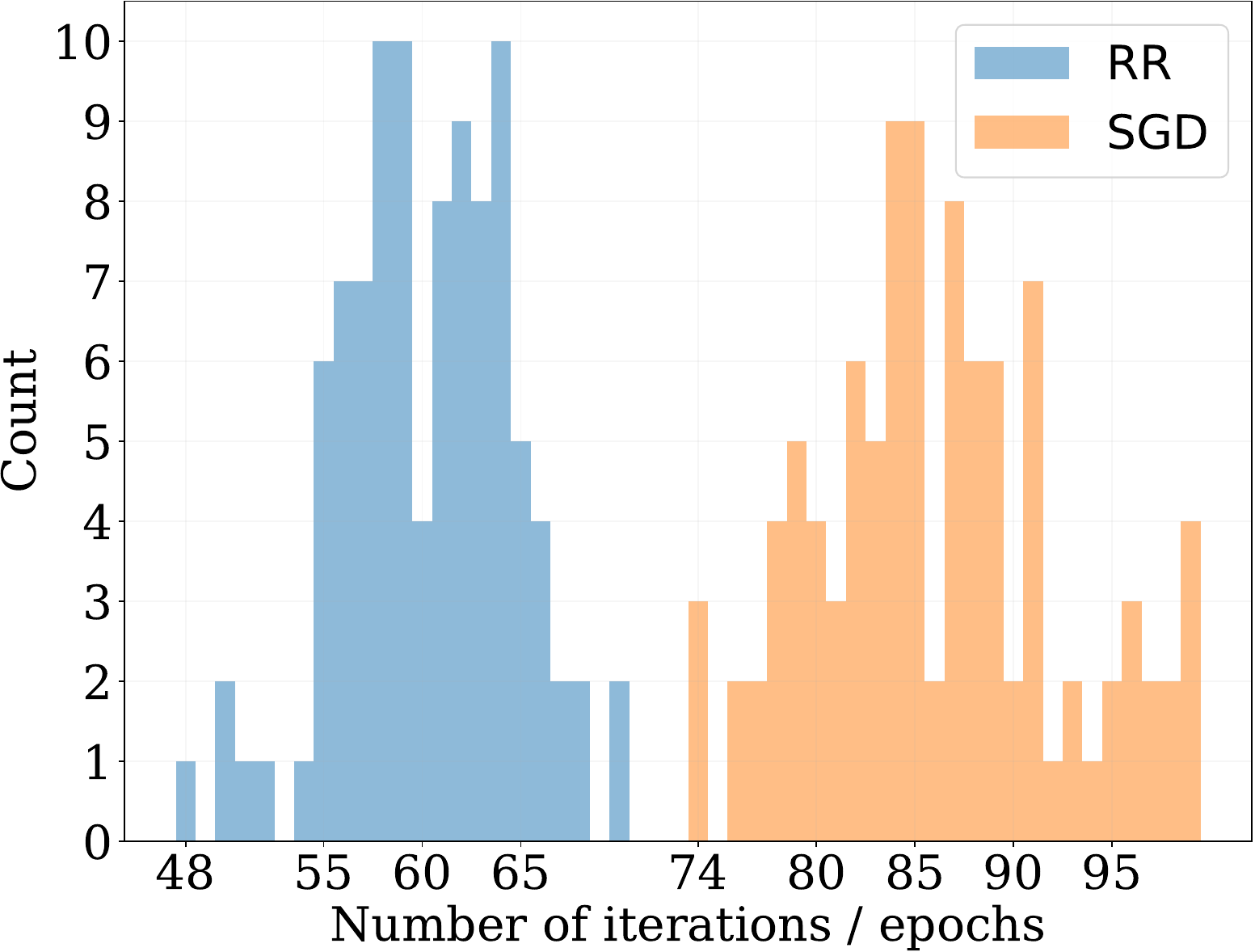}
	\caption{ {Histogram of iterations / epochs $t$ for achieving $\|\nabla f(x_t)\|\leq 10^{-3}$ over 100 independent runs.}}
	\label{fig:hist_1e-3}
    \end{subfigure}
    \caption{ {Statistics of epochs $t$ of $\RR$ and SGD for achieving an $\varepsilon$-stationary point (i.e., $\|\nabla f(x_t)\|\leq \varepsilon$) with varying $\varepsilon$.\label{fig:hist_statistics}}}
\end{figure}

Next, we display the gradient norm statistics for the last iterate in both algorithms in \Cref{fig:lft}. It can observed that $\RR$ not only tends to yield a smaller gradient norm of the last iterate, but also exhibits a superior concentration property. This empirical observation aligns with our theoretical findings that the gradient norm in $\RR$ converges with high probability.  In \Cref{fig:losstest}, we show the training loss and test accuracy of $\RR$ and SGD. We can conclude that $\RR$ provides a slightly smaller training loss and demonstrates a slightly superior test accuracy.

\begin{figure}[t]
    \centering
    \includegraphics[width=0.6\linewidth]{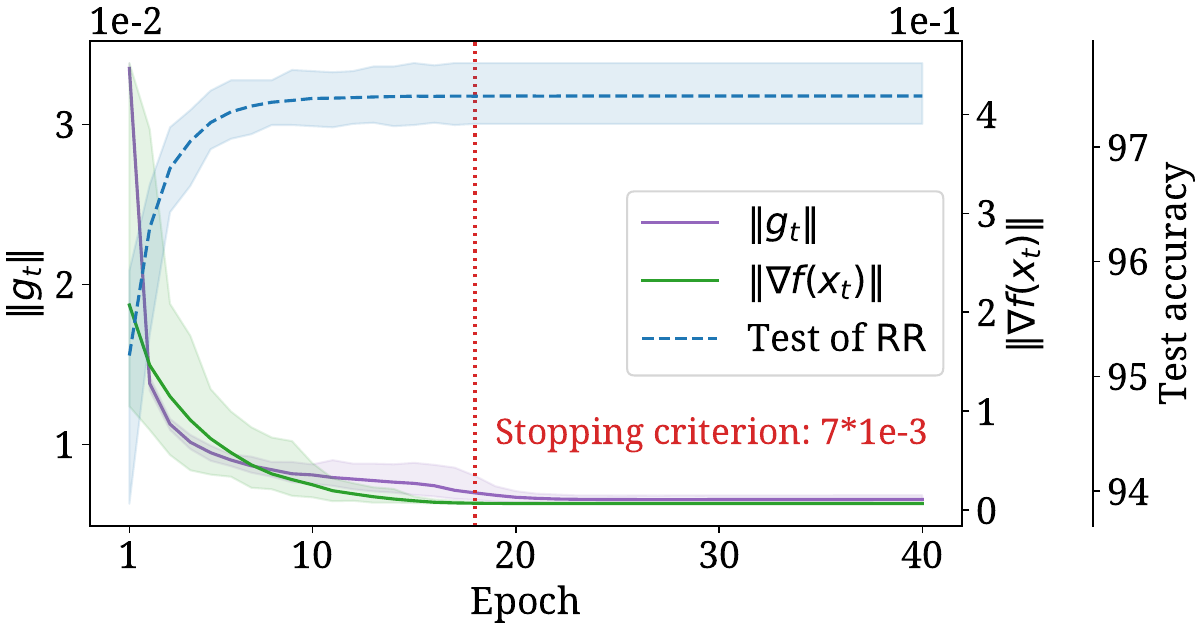}
    \caption{Evolution of $\|g_t\|$, $\|\nabla f(x_t)\|$, and test accuracy of $\RR$.}
    \label{fig:gttest}
\end{figure}

 {Our theory, together with the existing complexity comparison between $\RR$ and SGD, indicates that $\RR$ can take less steps in the sense of high probability for converging to an $\varepsilon$-stationary point compared to SGD. To support this result, we conduct experiments to display the histograms of epochs $t$ of $\RR$ and SGD for achieving $\|\nabla f(x_t)\|\leq \varepsilon$ with $\varepsilon = 10^{-2}$ and $\varepsilon = 10^{-3}$. The results are depicted in \Cref{fig:hist_statistics}. Note that we change the step decay factor $0.7$ to $0.95$ for the setting where $\varepsilon = 10^{-3}$, as otherwise the step size will be decayed to nearly $0$ too early. It can be observed from \Cref{fig:hist_statistics} that $\RR$ often takes less steps to find a target solution than SGD. Additionally, $\RR$'s superiority tends to be clearer when $\varepsilon$ is smaller. These experiment observations align with the theory.}

In addition, we conduct experiments to study the stopping criterion $\|g_t\|\leq \eta \varepsilon$ defined in \eqref{eq:stopping criterion}. The result is displayed in \Cref{fig:gttest}. We observed that $\|g_t\|$ finally aligns with $\|\nabla f(x_t)\|$ after $20$ epochs, aligning with our stopping criterion theorems.  It is also demonstrated that $\|g_t\|$ decreases along with the epoch index $t$. Upon setting the stopping criterion in \eqref{eq:stopping criterion} to $\|g_t\|\leq 7\times 10^{-3}$, the training process completes around the $17$th epoch, yielding a converged test accuracy. This suggests that $\|g_t\|$ is a practical measure that can be used as a stopping criterion.

\section{Conclusion and Discussions}
In this work, we established high probability complexity guarantees for $\RR$. In particular, we derived a high probability ergodic sample complexity guarantee for finding a stationary point, without using additional assumptions beyond those in the traditional in-expectation analysis of $\RR$. Furthermore, we proposed a stopping criterion for $\RR$ (denoted as $\RRsc$). With a designed horizon-free step size rule, such a stopping criterion terminates the method after a finite number of iterations and returns a last iterate with the same order of complexity guarantee.
The key to establishing these high probability results is a new concentration property we derived for random reshuffling, which could be of independent interest. 

The dependence on $n$ in \eqref{eq:RR sampling gradients} is caused by bounding the random variable in \eqref{eq:variance bound rv} using variance. While it does not affect our complexity, improving $n$ to $i$ (if possible) could be insightful. Additionally, with the derived concentration property for random reshuffling, it would be interesting to establish a tight complexity bound for $\RR$ to avoid saddle points, which is expected to be the same order as the first-order complexity derived in this work. 
We leave these areas for future exploration.

\bibliographystyle{plain}
\bibliography{references}

\end{document}